\documentclass{amsart}

\usepackage[a4paper,hmargin=3cm,vmargin=4cm]{geometry}
\usepackage{amsfonts,amssymb,amscd,amstext}
\usepackage{mathtools}
\mathtoolsset{showonlyrefs=true} 
\usepackage[utf8]{inputenc}
\usepackage[colorlinks=true,linkcolor=blue,citecolor=red]{hyperref}
\usepackage{braket,relsize,nccmath}

\renewcommand{\leq}{\leqslant}
\renewcommand{\geq}{\geqslant}
\newcommand{\ptl}{\partial}
\newcommand{\rr}{{\mathbb{R}}}
\newcommand{\rrn}{\mathbb{R}^{n+1}}
\newcommand{\la}{\lambda}
\newcommand{\sph}{\mathbb{S}}
\newcommand{\nn}{\mathbb{N}}
\newcommand{\sub}{\subset}
\newcommand{\subeq}{\subseteq}
\newcommand{\escpr}[1]{\big<#1\big>}
\newcommand{\Sg}{\Sigma} 
\newcommand{\sg}{\sigma}
\newcommand{\Om}{\Omega}
\newcommand{\eps}{\varepsilon}
\newcommand{\var}{\varphi}
\newcommand{\ric}{\text{Ric}}
\newcommand{\ind}{\mathcal{Q}}
\newcommand{\indo}{\mathcal{I}}

\DeclareMathOperator{\divv}{div}

\newtheorem{theorem}{Theorem}[section]
\newtheorem{proposition}[theorem]{Proposition}
\newtheorem{lemma}[theorem]{Lemma}
\newtheorem{corollary}[theorem]{Corollary}

\theoremstyle{definition}

\newtheorem{remark}[theorem]{Remark}
\newtheorem{remarks}[theorem]{Remarks}
\newtheorem{example}[theorem]{Example}

\theoremstyle{remark}

\numberwithin{equation}{section}

\begin{document}

\title[Stable and isoperimetric regions in weighted manifolds with boundary]{Stable and isoperimetric regions in some weighted manifolds with boundary}

\author[C.~Rosales]{C\'esar Rosales}
\address{Departamento de Geometr\'{\i}a y Topolog\'{\i}a and Excellence Research Unit ``Modeling Nature'' (MNat) Universidad de Granada, E-18071,
Spain.} 
\email{crosales@ugr.es}

\date{\today}

\thanks{The author is supported by MINECO grant No.~MTM2017-84851-C2-1-P and Junta de Andaluc\'ia grant No.~FQM325} 



\begin{abstract}
In a Riemannian manifold with a smooth positive function that weights the associated Hausdorff measures we study stable sets, i.e., second order minima of the weighted perimeter under variations preserving the weighted volume. By assuming local convexity of the boundary and certain behaviour of the Bakry-\'Emery-Ricci tensor we deduce rigidity properties for stable sets by using deformations constructed from parallel vector fields tangent to the boundary. As a consequence, we completely classify the stable sets in some Riemannian cylinders $\Om\times\rr$ with product weights. Finally, we also establish uniqueness results showing that any minimizer of the weighted perimeter for fixed weighted volume is bounded by a horizontal slice $\Om\times\{t\}$.
\end{abstract}

\maketitle

\thispagestyle{empty}

\section{Introduction}
\label{sec:intro}
\setcounter{equation}{0}

A \emph{weighted manifold} is a triple $(M,g,f)$, where $(M,g)$ is a Riemannian manifold, possibly with non-empty boundary $\ptl M$, and $f$ is a smooth positive function used to weight the Hausdorff measures associated to the Riemannian distance. In this context, the \emph{partitioning problem} seeks those sets in $M$ minimizing the weighted perimeter for fixed weighted volume. Here we consider the \emph{interior perimeter}, so that the contribution of the boundary intersected with $\ptl M$ is not taken into account. When the solutions to this problem exist they are called \emph{weighted isoperimetric regions} or \emph{weighted minimizers}. In the study of these regions the analysis of critical points and second order minima plays an important role. We say that a set is \emph{weighted stable} if it is a critical point with non-negative second derivative of the weighted perimeter under deformations preserving the weighted volume and the boundary $\ptl M$. Our aim in this work is to prove rigidity properties and classification results for weighted stable sets in order to deduce uniqueness of weighted minimizers.

Due to the regularity of weighted isoperimetric regions, see Theorem~\ref{th:regularity}, we can restrict to weighted stable sets $E$ having finite weighted perimeter and almost smooth interior boundary $\overline{\ptl E\cap\text{int}(M)}$. Moreover, by Theorem~\ref{th:parabolicity}, we can suppose that the regular part $\Sg$ of the interior boundary verifies the \emph{weighted parabolicity condition}, an analytical feature that generalizes the compactness of $\Sg$, see Section~\ref{subsec:isoperimetric} for precise definitions and examples. After previous work of Bayle~\cite[Ch.~3]{bayle-thesis} when $\ptl\Sg=\emptyset$, see also~\cite[Sect.~3]{rcbm}, Castro and the author~\cite[Sect.~3]{castro-rosales} computed the first and second variational formulas to derive basic consequences for weighted stable sets with $\ptl \Sg\neq\emptyset$. It follows that the \emph{weighted mean curvature} $H_f$ given in \eqref{eq:fmc} is constant, $\Sg$ meets $\ptl M$ orthogonally along the free boundary $\ptl\Sg$, and the \emph{weighted index form} $\indo_f$ introduced in \eqref{eq:index1} satisfies $\indo_f(u,u)\geq 0$ for any function $u\in C^\infty_0(\Sg)$ such that $\int_\Sg u\,da_f=0$. This inequality involves the second fundamental form $\text{II}$ of $\ptl M$ and the \emph{Bakry-\'Emery-Ricci tensor} $\ric_f$ defined in \eqref{eq:fricci} in such a way that it becomes more restrictive when \emph{$\ptl M$ is locally convex} ($\text{II}\geq 0$) and $\ric_f\geq 0$. Indeed, these convexity assumptions together with the weighted parabolicity of $\Sg$ lead to an \emph{extended stability inequality} $\indo_f(u,u)\geq 0$, which is valid under certain integrability hypotheses for mean zero functions that need not vanish on $\Sg$, see Proposition~\ref{prop:extended}. Thus, the idea is to insert a suitable test function into this inequality to infer interesting properties for stable sets. 

In order to describe in more detail our setting and results we first recall some previous works about the partitioning problem in weighted manifolds with locally convex boundary and $\ric_f\geq 0$. 

In Gauss space, which is $\rrn$ with weight $\gamma_c(p):=e^{-c|p|^2/2}$, $c>0$, it is well known that, up to sets of volume zero, the Euclidean half-spaces uniquely minimize the weighted perimeter for fixed weighted volume, see Morgan~\cite[Thm.~18.2]{gmt} and the references therein. Indeed, the half-spaces are also the only stable sets with respect to $\gamma_c$, as was established in \cite{rosales-gauss}. In the same paper, the author showed that, in a Gaussian half-space or slab $M$, any weighted stable set is the intersection with $M$ of a half-space parallel or perpendicular to $\ptl M$. From here, an area comparison between the stable candidates allows to conclude that only those orthogonal to $\ptl M$ minimize. 

In \cite{bcm}, Brock, Chiacchio and Mercaldo employed optimal transport to deduce that in $\rr^n\times\rr^+$ with weight $f(x,t):=t^m\,\gamma_1(x,t)$, $m>0$, the half-spaces perpendicular to the boundary hyperplane $t=0$ are the unique weighted isoperimetric regions. This situation is a particular case of a one-dimensional log-concave perturbation $f(x,t):=e^{\omega(t)}\,\gamma_c(x,t)$ of $\gamma_c$. We studied these weights in \cite{rosales-gauss} for Euclidean half-spaces and slabs $\rr^n\times (a,b)$, proving the same uniqueness results for stable and isoperimetric sets as in the Gaussian setting. 

More generally, it is interesting to investigate the partitioning problem for a weight $f(p):=e^{\omega(p)}\,\gamma_c(p)$, where $\omega$ is a concave function. A reason for this is a result of Caffarelli~\cite{caffarelli}, see also Kim and Milman~\cite{contraction}, saying that the optimal transport Brenier map pushing the Gaussian probability measure forward the probability measure associated to $f(p)\,dp$ is a $1$-Lipschitz function. By combining this fact with the Gaussian isoperimetric inequality, Brock, Chiacchio and Mercaldo~\cite{bcm3} discussed the partitioning problem for certain perturbations of $\gamma_1$ that affect not just one but $n$ variables, inside the cylinder $M:=C\times\rr$ where $C:=\prod_{i=1}^n(a_i,b_i)$ and $-\infty\leq a_i<b_i\leq \infty$ for any $i=1,\ldots,n$. This applies in particular for perturbation terms $\omega(x_1,\ldots,x_n):=\sum_{i=1}^n\omega_i(x_i)$ depending on $C^2$ concave functions $\omega_i$ on $(a_i,b_i)$. We stress that the gradient estimate for the Brenier map used in \cite{bcm3} was obtained by means of elementary and self-contained tools.

Other relevant generalizations of the Gauss space are the \emph{shrinking gradient Ricci solitons}. These are weighted manifolds $(M,g,f)$ for which the Bakry-\'Emery-Ricci tensor verifies the equality $\ric_f=c\,g$ for some constant $c>0$. The interest in these spaces comes from the theory of the Ricci flow since they provide self-similar solutions that model some singularities of the flow, see \cite{hamilton}. From the geometric viewpoint they are weighted counterparts of the Einstein manifolds. Indeed, for any Einstein manifold $\Om$ of constant Ricci curvature $c>0$, the product manifold $\Om\times\rr^k$ with vertical weight $f(x,t):=e^{-c|t|^2/2}$ is a $c\,$-shrinker. In relation to our problem, stability properties of complete hypersurfaces with constant weighted mean curvature, finite weighted area and \emph{empty boundary}, including estimates on the index of $\indo_f$ and characterization theorems for low indexes, have been proved by Colding and Minicozzi~\cite{colding-minicozzi} and by McGonagle and Ross~\cite{mcgonagle-ross} in Gauss space, by Cheng, Mejia and Zhou~\cite{cmz2} in $\sph^n\times\rr$, and by Alencar and Rocha~\cite{alencar-rocha} in shrinkers with a non-trivial parallel vector field. We remark that these results do not apply directly to the interior boundary of a weighted isoperimetric region due to the possible presence of singularities in high dimensions. We also point out that a shrinker having a non-vanishing parallel vector field is rigid, as stated by Petersen and Wylie~\cite{petersen-wylie}. This means that it is isometric to a Riemannian cylinder $\Om\times\rr$ with a product weight $f(x,t):=e^{h(x)}\,e^{v(t)}$ such that $\Om$ is a $c\,$-shrinker with respect to the horizontal factor $e^h$, while the vertical one $e^v$ is an affine perturbation of $\gamma_c$.

Motivated by all this, in this paper we focus on the analysis of weighted stable sets in a broader context. More precisely, we consider a Riemannian cylinder $M:=\Om\times\rr$ over a complete manifold $\Om$, possibly with smooth locally convex boundary $\ptl\Om$, and endowed with a product weight $f(x,t):=e^{h(x)}\,e^{v(t)}$ such that $\ric_h\geq c>0$ on $\Omega$ and $e^v$ is an affine perturbation of the Gaussian weight $\gamma_c$. Observe that our generalization comes not only since we allow arbitrary (smooth) cylinders with non-empty boundary, but also since the condition $\ric_h\geq c$ permits arbitrary horizontal log-concave perturbations of a $c\,$-shrinker, i.e., weights of the form $e^{\mu(x)+\omega(x)}\,e^{v(t)}$ where $\ric_\mu=c$ and $\omega$ is a concave function. In particular, any smooth convex cylinder with Gaussian weight $\gamma_c$ eventually perturbed by a log-concave horizontal function is included into this setting.

The weighted cylinders described above are concrete cases of a Riemannian manifold $(M,g)$ having a unit parallel vector field $X$ tangent to the boundary $\ptl M$, and endowed with a weight $f$ such that the Bakry-\'Emery-Ricci tensor ``splits with respect to $X$'', in the sense that $\ric_f(X,Y)=c\,g(X,Y)$ when $Y$ is any vector field, and $\ric_f(Y,Y)\geq c\,|Y|^2$ when $Y$ is orthogonal to $X$. Under these more general conditions we establish in Section~\ref{sec:stable} rigidity properties for any weighted stable set $E$. Indeed, in Theorem~\ref{th:stable} we show that, either the vector field $X$ is tangent to the regular part $\Sg$ of the interior boundary of $E$, or $\Sg$ is a totally geodesic hypersurface in $(M,g)$ with $\ric_f(N,N)=c$ on $\Sg$ and $\text{II}(N,N)=0$ along $\ptl\Sg$. As we noticed before, we allow the presence of a negligible singular set $\Sg_0$ in the interior boundary, so that our stability result applies for weighted isoperimetric regions. To prove the theorem we employ the extended stability inequality in Proposition~\ref{prop:extended} (ii) with the test function $u:=\alpha+g(X,N)$, where $N$ is the unit normal on $\Sg$ pointing into $E$ and $\alpha$ is a constant such that $\int_\Sg u\,da_f=0$.  From a geometric viewpoint this function comes from a volume-preserving variation which combines equidistant sets to $E$ with the one-parameter flow associated to $X$. After a long computation, where we use that $\var:=g(X,N)$ is an eigenfunction for the stability operator $\mathcal{L}_f$ defined in \eqref{eq:jacobi}, we infer that this deformation strictly decreases the weighted area unless the hypersurface $\Sg$ satisfies the announced restrictions.

Next, we come back to Riemannian cylinders $M:=\Om\times\rr$ with weights. In $M$ the vertical vector field $\xi$ determines a foliation by \emph{horizontal slices} $\Om_t:=\Om\times\{t\}$. For a product weight $f(x,t):=e^{h(x)}\,e^{v(t)}$ these slices are totally geodesic hypersurfaces of constant weighted mean curvature and meeting $\ptl M$ orthogonally. Moreover, by assuming as above that $\ric_h\geq c>0$ and $e^v$ is an affine perturbation of $\gamma_c$, it follows that the associated \emph{horizontal half-spaces} $\Om\times(-\infty,t)$ and $\Om\times (t,\infty)$ are weighted stable sets, see Examples~\ref{ex:horstable} and \ref{ex:starig}. In Section~\ref{subsec:unique} we analyze weighted stable sets in this situation. In Theorem~\ref{th:stacyl} we refine the conclusions of Theorem~\ref{th:stable}, as we establish that the hypersurface $\Sg$ of a weighted stable set $E$ is always totally geodesic (even if $\xi$ is tangent to $\Sg$) and satisfies the equalities $\ric_f(N,N)=c$ on $\Sg$ and $\text{II}(N,N)=0$ along $\ptl\Sg$. Note that these rigidity properties allow the existence of weighted stable sets different from horizontal half-spaces: in a Gaussian slab $M$ any half-space with boundary parallel to $\ptl M$ is weighted stable. So, in order to deduce the uniqueness of horizontal half-spaces as weighted stable sets we need further hypotheses involving the cylinder and the weight. In Corollary~\ref{cor:staunique} we show two conditions in this direction. As a matter of fact, the conclusions $\ric_f(N,N)=c$ and $\text{II}(N,N)=0$ along $\ptl\Sg$ enforce the hypersurface $\Sg$ to be a horizontal slice if we accept that, either $\ric_h>c$ on $\Om$, or $\ptl\Om$ is locally strictly convex and $\ptl\Sg\neq\emptyset$. We remark that the hypothesis $\ptl\Sg\neq\emptyset$ is necessary, since there are examples of weighted stable sets with $\ptl\Sg=\emptyset$ even if $\ptl M\neq\emptyset$.

In Section~\ref{subsec:unique} we also study the isoperimetric regions in our weighted cylinders. Observe that in such a cylinder $M$ we have $\ric_f\geq c>0$ and so, $M$ has finite weighted volume by a result of Morgan~\cite{morgandensity}. Thus, we can use the L\'evy-Gromov type isoperimetric inequality of Bakry and Ledoux~\cite{bl}, see also Milman~\cite{milman}, to conclude that, up to normalization, the weighted perimeter of a set in $M$ is no lower than the perimeter with respect to $\gamma_c$ of a Euclidean half-space enclosing the same weighted volume. By combining this fact with the Gaussian isoperimetric inequality, it is easy to prove that the horizontal half-spaces in $M$ are always weighted minimizers. In Corollary~\ref{cor:isocyl} we employ our stability analysis to infer restrictions for arbitrary minimizers and criteria for horizontal half-spaces to be the only weighted isoperimetric regions. In Appendix~\ref{app:a} we provide a self-contained proof of the isoperimetric property of horizontal half-spaces. The argument is mainly based on the second variation formula and also allows to deduce Corollary~\ref{cor:isocyl} without having in mind deep stability consequences.

In Section~\ref{subsec:examples} we gather some relevant examples where our results are applied. These cover most of the cases previously considered and new interesting situations. Sometimes we find vertical stable candidates that must be discarded as minimizers by means of additional comparisons. This happens in the $c\,$-shrinker given by the weighted cylinder $M:=\Om\times\rr$, where $\Om$ is an Einstein manifold of Ricci curvature $c>0$, and the weight has the form $f(x,t):=e^{-ct^2/2}$. In this example, the uniqueness of horizontal half-spaces as weighted isoperimetric regions follows from the classification of weighted stable sets in Corollary~\ref{cor:stasol} after a geometric inequality of Maeda~\cite{maeda} showing that the vertical candidates do not minimize. In a similar way, in a Euclidean half-space or slab $M$ endowed with an arbitrary horizontal log-concave perturbation of the Gaussian weight, we rule out half-spaces parallel to $\ptl M$ as minimizers by taking into account the Heintze-Karcher type inequality proved by Morgan~\cite{morgandensity}.

We finish this introduction by pointing out that many of our results still hold, with suitable modifications, when the horizontal weight $e^h$ verifies $\ric_h\geq 0$  and the vertical one $e^v$ is a log-linear function. However, this requires the further hypothesis $\int_\Sg g(\xi,N)\,da_f\neq 0$, which is very restrictive. A detailed discussion is found in Section~\ref{subsec:zerocase}.

The paper contains four sections. In Section~\ref{sec:prelimi} we introduce the notation and review preliminary facts about weighted cylinders, isoperimetric regions and stable sets. In Section~\ref{sec:stable} we prove our main stability result for manifolds with locally convex boundary and certain behaviour of the Bakry-\'Emery-Ricci tensor with respect to a parallel vector field. The fourth section is devoted to our rigidity properties and uniqueness criteria for stable sets and minimizers in weighted cylinders.

\section{Preliminaries}
\label{sec:prelimi}
\setcounter{equation}{0}

In this section we introduce the notation and state some results for weighted Riemannian cylinders, isoperimetric regions and stable sets that will be useful throughout this work. We begin with some generalities about manifolds with weights.

A \emph{weighted manifold} is a triple $(M,g,f)$, where $(M,g)$ is a smooth complete Riemannian manifold of dimension $n+1$, possibly with non-empty smooth boundary $\ptl M$, and $f\in C^\infty(M)$ with $f>0$ on $M$. We denote by $\text{int}(M)$ the set $M\setminus\ptl M$. The notations $\escpr{\cdot\,,\cdot}$ and $|\cdot|$ refer to the scalar product and the associated norm for tangent vectors in $M$. 

The function $f$ is used to weight the Hausdorff measures associated to the Riemannian distance. In particular, for any Borel set $E\subeq M$, the \emph{weighted volume} and the \emph{interior weighted area} of $E$ are defined by
\begin{equation}
\label{eq:volarea}
V_f(E):=\int_E \,dv_f,\qquad 
A_f(E):=\int_{E\cap\text{int}(M)} da_f, 
\end{equation}
where $dv_f:=f\,dv$ and $da_f:=f\,da$ are the weighted elements of volume and area, respectively. Note that $E\cap\ptl M$ does not contribute to $A_f(E)$.

By following the approach of Caccioppoli and De Giorgi, the \emph{interior weighted perimeter} of a Borel set $E\subeq M$ is introduced by equality
\[
P_f(E):=\sup\left\{\int_E\divv_f X\,dv_f\,;\,|X|\leq 1\right\},
\]
where $X$ ranges over smooth vector fields with compact support on $\text{int}(M)$. The integrand at the right hand side of the equation is the \emph{weighted divergence} given by
\begin{equation}
\label{eq:divf}
\divv_f X:=\divv X+\escpr{\nabla\psi,X}.
\end{equation}
Here $\divv$ is the divergence in $(M,g)$ and $\nabla\psi$ stands for the Riemannian gradient of $\psi:=\log(f)$. The divergence theorem in $(M,g)$ and the identity $\divv_fX\,dv_f=\divv(fX)\,dv$ imply
\begin{equation}
\label{eq:macorra}
P_f(E)=A_f\big(\ptl E\cap\text{int}(M)\big),
\end{equation}
for any open set $E$ in $M$ such that $\ptl E\cap\text{int}(M)$ is a smooth hypersurface. By a set of \emph{finite weighted perimeter} in $M$ we mean a Borel set $E\subeq M$ such that $P_f(E)$ is finite. As the perimeter does not change by sets of volume zero we can suppose that $0<V_f(E\cap B)<V_f(B)$ for any open metric ball $B$ centered at $\ptl E\cap\text{int}(M)$, see \cite[Prop.~3.1]{giusti}. 

In $(M,g,f)$ there are weighted notions of curvature involving the Riemannian curvatures and the derivatives of the function $\psi:=\log(f)$. Here we will only consider the \emph{Bakry-\'Emery-Ricci tensor}, first introduced by Lichnerowicz \cite{lich1,lich2} and later generalized by Bakry and \'Emery \cite{be} in the framework of diffusion generators. This is given by
\begin{equation}
\label{eq:fricci}
\text{Ric}_f:=\ric-\nabla^2\psi,
\end{equation}
where $\text{Ric}$ and $\nabla^2$ denote, respectively, the Ricci tensor and the Hessian operator in $(M,g)$. The Ricci tensor is the $2$-tensor $\ric(w_1,w_2):=\text{trace}(u\mapsto R(w_1,u)w_2)$, where $R$ is the curvature tensor in $(M,g)$ as defined in \cite[Sect.~4.2]{dcriem}. For a point $p\in M$, the \emph{Bakry-\'Emery-Ricci curvature} in the direction of a unit vector $w\in T_pM$ is the number $(\ric_f)_p(w,w)$. If this is always greater than or equal to a constant $c\in\rr$, then we write $\ric_f\geq c$. 

If the equality $\ric_f=c\,g$ is satisfied for some $c\in\rr$, then $(M,g,f)$ is said to be a \emph{$c$-gradient Ricci soliton}. The soliton is a \emph{shrinker} (resp. an \emph{expander}) when $c>0$ (resp. $c<0$). A \emph{steady soliton} is one for which $\ric_f=0$. This terminology comes from the theory of the Ricci flow, in relation to the study of its singularities. It is clear that an Einstein manifold of constant Ricci curvature $c$ is a gradient Ricci soliton with respect to constant weights. Another example is the \emph{$c$-Gaussian soliton}, which is Euclidean space $\rrn$ with the radial weight $\gamma_c(p):=e^{-c|p|^2/2}$.

In the sequel, we will denote by $C_0^\infty(M)$ the space of smooth functions with compact support in $M$ (possibly non-vanishing along $\ptl M$), by $L^1(M,dv_f)$ the space of integrable functions with respect to the measure $dv_f$, and by $H^1(M,dv_f)$ the weighted Sobolev space of functions $u:M\to\rr$ such that $u^2\in L^1(M,dv_f)$ and $u$ has a distributional gradient $\nabla u$ with $|\nabla u|^2\in L^1(M,dv_f)$.

\subsection{Weighted cylinders}
\label{subsec:cylinders}
\noindent

For a complete $n$-dimensional Riemannian manifold $\Om$, possibly with boundary $\ptl\Om$, the \emph{Riemannian cylinder} over $\Om$ is the Riemannian manifold $(M,g)$, where $M:=\Om\times\rr$ and $g$ is the product of the Riemannian metric in $\Om$ with the standard metric in $\rr$. The reader is referred to \cite[Ch.~7]{oneill} for geometric properties of Riemannian cylinders and general warped products.

The \emph{vertical vector field} $\xi$ in $M$ is the one that assigns to any point $p=(x,t)\in M$ the tangent vector $\xi(p):=(0,1)\in T_x\Om\times\rr$. This is a unit parallel vector field which is tangent to $\ptl M=\ptl\Om\times\rr$. The isometries in the one-parameter group associated to $\xi$ are called \emph{vertical translations}. We say that a submanifold $\Sg$ in $M$ is \emph{vertical} if $\xi$ is tangent to $\Sg$. The vector field $\xi$ determines a foliation of $M$ by  \emph{horizontal slices} $\Om_t:=\Om\times\{t\}$, $t\in\rr$, for which $\xi$ is a unit normal vector. Hence $\Om_t$ is a totally geodesic hypersurface in $M$ meeting orthogonally $\ptl M$ along $\ptl\Om_t=\ptl\Om\times\{t\}$. The (open) \emph{horizontal half-spaces} determined by $\Om_t$ are the sets $\Om\times(-\infty,t)$ and $\Om\times(t,\infty)$. 

Suppose that there is a smooth unit normal vector field along $\ptl\Om$. Then, we can use the vertical translations to extend it to a smooth unit normal along $\ptl M$. In this case, we denote by $\text{II}$ the second fundamental form of $\ptl M$ with respect to the inner unit normal. It is clear that $\text{II}(\xi,\xi)=0$ along $\ptl M$. Observe also that $\ptl M$ is \emph{locally convex}, i.e. $\text{II}\geq 0$, if and only if $\ptl\Om$ is locally convex. 

The most natural weights to consider in a Riemannian cylinder $M=\Om\times\rr$ are product weights. These are defined as $f(x,t):=e^{h(x)}\,e^{v(t)}$ for some functions $h\in C^\infty(\Om)$ and $v\in C^\infty(\rr)$. In particular, $\Om$ and $\rr$ become weighted manifolds with respect to the  horizontal and vertical components of the weight. For later use we note that, for any point $p=(x,t)\in M$, and any pair of vectors $w_1,w_2\in T_pM$, the Bakry-\'Emery-Ricci tensor satisfies
\begin{equation}
\label{eq:riccidesc}
(\ric_f)_p(w_1,w_2)=(\ric_h)_x\big((w_1)_*,(w_2)_*\big)-v''(t)\,\escpr{\xi(p),w_1}\,\escpr{\xi(p),w_2},
\end{equation}
where $(w_i)_*\in T_x\Om$ is the horizontal projection of $w_i$ and $\ric_h$ stands for the Bakry-\'Emery-Ricci tensor in $\Om$ with respect to the weight $e^h$.

As a consequence of \eqref{eq:riccidesc}, if $\Om$ is a $c$-gradient Ricci soliton with respect to a weight $e^h$, then $\Om\times\rr$ is a $c$-gradient Ricci soliton for the product weight $f(x,t):=e^{h(x)}\,e^{-ct^2/2}$. The converse is true by a result of Petersen and Wylie~\cite[Lem.~2.1]{petersen-wylie}: a weight $f$ on $\Om\times\rr$ defines a $c$-gradient Ricci soliton if and only if $f$ is a product weight such that the resulting horizontal and vertical weighted manifolds are also $c$-gradient Ricci solitons. A relevant case is the weighted cylinder $\Om\times\rr$, where $\Om$ is an Einstein manifold of constant Ricci curvature $\ric_\Om$ equal to $c$ and $f(x,t):=e^{-ct^2/2}$.

\subsection{Isoperimetric regions}
\label{subsec:isoperimetric}
\noindent

The \emph{partitioning problem} in a weighted manifold $(M,g,f)$ studies those sets in $M$ that minimize the weighted perimeter for fixed weighted volume. A \emph{weighted isoperimetric region} or \emph{weighted minimizer} of volume $w\in(0,V_f(M))$ is a set of finite weighted perimeter $E\sub M$ with $V_f(E)=w$ and such that $P_f(E)\leq P_f(E')$, for any other set $E'\sub M$ with $V_f(E')=w$.

In our context, where the weight $f$ is smooth and positive in $M$, the regularity properties of weighted minimizers are the same as in the unweighted case. In the Euclidean setting, the interior regularity was obtained by Gonzalez, Massari and Tamanini \cite{go-ma-ta}, while Gr\"uter~\cite{gruter} proved the regularity along the free boundary. The reader is referred to Morgan~\cite[Sect.~3.10]{morgan-reg} and Milman~\cite[Sects.~2.2, 2.3]{milman} for the corresponding extensions to Riemannian manifolds with or without weights. We gather their results in the next statement.

\begin{theorem}
\label{th:regularity}
Let $E$ be a weighted isoperimetric region in a weighted manifold $(M,g,f)$ of dimension $n+1$. Then, the interior boundary $\overline{\ptl E\cap\emph{int}(M)}$ is a disjoint union $\Sg\cup\Sg_0$, where $\Sg$ is a smooth embedded hypersurface, possibly with boundary $\ptl\Sg=\Sg\cap\ptl M$, and $\Sg_0$ is a closed set of singularities with Hausdorff dimension less than or equal to $n-7$. Moreover, at any point $p\in\Sg_0$, there is a closed area-minimizing tangent cone $C_p\sub T_pM$ different from a hyperplane or a half-hyperplane. In particular, the squared norm $|\sg|^2$ of the second fundamental form of $\Sg$ tends to $\infty$ when we approach $p$ from $\Sg$.
\end{theorem}

Observe that neither the weighted minimizers nor their interior boundaries need to be bounded. For example, in $c\,$-Gaussian solitons with $c>0$ any weighted isoperimetric region coincides, up to a set of volume zero, with a half-space~\cite[Ch.~18]{gmt}. However, for bounded weights, any minimizer satisfies an analytical generalization of compactness: the parabolicity of its interior boundary. 

The parabolicity of a Riemannian manifold is characterized by the fact that compact subsets have vanishing capacities, see for instance~\cite[Sect.~5]{grigoryan}. By following this approach, we say that a smooth hypersurface $\Sg$ of a weighted manifold $(M,g,f)$ is \emph{weighted parabolic} or has \emph{null weighted capacity} if $\text{Cap}_f(K)=0$ for any compact set $K\subeq\Sg$. Here $\text{Cap}_f(K)$ is the \emph{weighted capacity} of $K$ defined as in Grigor'yan and Masamune \cite[Sect.~2]{gri-masa} by 
\[
\text{Cap}_f(K):=\inf\left\{\int_\Sg|\nabla_\Sg u|^2\,da_f\,;\,u\in C^\infty_0(\Sg),\,0\leq u\leq 1,\,u=1\text{ in }K\right\},
\]
where $\nabla_\Sg u$ denotes the Riemannian gradient of $u$ in $\Sg$. From this definition it is easy to see, by taking a countable exhaustion of $\Sg$ by relatively compact subsets, that $\Sg$ is weighted parabolic if and only if there is a sequence $\{\var_k\}_{k\in\nn}\sub C^\infty_0(\Sg)$ with $\var_k\neq 0$ and such that
\begin{equation}
\label{eq:sequence}
\begin{cases}
\mathsmaller{\bullet} \ 0\leq\var_k\leq 1 \text{ for any } k\in\nn, 
\\
\mathsmaller{\bullet} \ \text{given a compact set } K\subeq\Sigma, \text{ there is } k_0\in\nn \text{ such that } \var_k=1 \text{ on } K \text{ for any } k\geq k_0,
\\
\mathsmaller{\bullet} \ \lim_{k\to\infty}\int_\Sg|\nabla_\Sg\var_k|^2\,da_f=0. 
\end{cases}
\end{equation}
It is clear that a compact hypersurface $\Sg$ is weighted parabolic. The converse is not true, and there are criteria ensuring parabolicity of non-compact hypersurfaces under additional conditions like growth properties of the weighted area, see Grigor'yan~\cite[Sect.~9.1]{grigoryan2}. In particular, any complete hypersurface $\Sg$ in $M$ with $A_f(\Sg)<\infty$ is weighted parabolic.

If $E$ is a weighted minimizer in $(M,g,f)$ with $\Sg_0=\emptyset$, then we know from Theorem~\ref{th:regularity} that the regular part $\Sg$ of $\overline{\ptl E\cap\text{int}(M)}$ is a closed subset of $M$ with $A_f(\Sg)<\infty$ and so, it is a weighted parabolic hypersurface. In the general case we have the following statement. 

\begin{theorem}
\label{th:parabolicity}
If $E$ is a weighted isoperimetric region in a weighted manifold $(M,g,f)$, and $f$ is bounded on $E$, then the regular part $\Sg$ of the interior boundary $\overline{\ptl E\cap\emph{int}(M)}$ is a weighted parabolic hypersurface.
\end{theorem}

The proof relies on the construction of a sequence $\{\var_k\}_{k\in\nn}$ as in \eqref{eq:sequence}. This requires perimeter estimates of $E$ inside metric balls and the fact derived from Theorem~\ref{th:regularity} that the $(n-2)$-dimensional Hausdorff measure of the singular set $\Sg_0$ vanishes. The result was obtained by Sternberg and Zumbrun \cite[Lem.~2.4]{sz} when $M$ is a Euclidean bounded set with constant weight. The reader is referred to Bayle~\cite[Prop.~2.5]{bayle-paper} for the case of unweighted compact manifolds. The extension to arbitrary manifolds with bounded weights follows from Bayle~\cite[p.~125]{bayle-thesis} and the author \cite[Thm.~3.6]{rosales-gauss}.

\subsection{Stability}
\label{subsec:stability}
\noindent

We introduce stable sets as second order minima of the interior perimeter under volume-preserving deformations. These provide the most natural candidates to be isoperimetric regions. Motivated by  Theorem~\ref{th:regularity} we will consider sets with almost smooth interior boundaries.

Let $(M,g,f)$ be a weighted manifold of dimension $n+1$. We take an open set $E\sub M$, whose interior boundary verifies $\overline{\ptl E\cap\text{int}(M)}=\Sg\cup\Sg_0$, where $\Sg$ is a smooth embedded hypersurface and $\Sg_0$ is a closed set of singularities with $A_f(\Sg_0)=0$. In case $\Sg$ has non-empty boundary we suppose that $\ptl\Sg=\Sg\cap\ptl M$. We denote by $N$ the unit normal on $\Sg$ pointing into $E$, and by $\nu$ the conormal vector of $\ptl\Sg$, i.e., the inner unit normal along $\ptl\Sg$ in $\Sg$. The fact that $E$ has almost smooth interior boundary implies equality \eqref{eq:macorra}, see~\cite[Ex.~12.7]{maggi}.

Let $X$ be a smooth vector field with compact support on $M$ and tangent along $\ptl M$. The associated \emph{variation} of $E$ is the family $E_s:=\phi_s(E)$, where $\{\phi_s\}_{s\in\rr}$ is the one-parameter flow of $X$. The variation is \emph{volume-preserving} if the volume functional $V_f(s):=V_f(E_s)$ is constant for any $s$ small enough. We say that $E$ is \emph{weighted stationary} if the perimeter functional $P_f(s):=P_f(E_s)$ satisfies $P_f'(0)=0$ for any volume-preserving variation. If, in addition, $P_f''(0)\geq 0$ for any volume-preserving variation, then we say that $E$ is \emph{weighted stable}. 

Observe that $\Sg_s:=\phi_s(\Sg)$ is a smooth hypersurface with $\ptl\Sg_s=\Sg_s\cap\ptl M$ and $A_f(\Sg_s)=P_f(s)$. Hence, if $E$ is weighted stable, then $\Sg$ is an \emph{$f$-stable hypersurface}, i.e., a second order minimum of the interior weighted area under volume-preserving variations. Therefore, we can deduce the following properties, see Castro and the author~\cite[Sect.~3]{castro-rosales}.

\begin{lemma}
\label{lem:varprop}
In the previous conditions on the weighted manifold $(M,g,f)$ and the set $E$, we have
\begin{itemize}
\item[(i)] If $E$ is weighted stationary, then $\Sg$ has constant weighted mean curvature and meets $\ptl M$ orthogonally along $\ptl\Sg$.
\item[(ii)] If $E$ is weighted stable, then $\indo_f(u,u)\geq 0$ for any function $u\in C^\infty_0(\Sg)$ with $\int_\Sg u\,da_f=0$.
\end{itemize}
Moreover, if $\Sg_0=\emptyset$, then the reverse statements are true.
\end{lemma}

Let us clarify some terminology and notation in the previous statement. The \emph{weighted mean curvature} of a hypersurface $\Sg$ with a smooth unit normal $N$ is the function $H_f$ introduced by Gromov~\cite[Sect.~9.4.E]{gromov-GAFA}, see also Bayle~\cite[Sect.~3.4.2]{bayle-thesis}, by equality
\begin{equation}
\label{eq:fmc}
H_f:=nH-\escpr{\nabla\psi,N},
\end{equation}
where $H$ is the Riemannian mean curvature of $\Sg$ with respect to $N$ and $\psi:=\text{log}(f)$. The \emph{weighted index form} of $\Sg$ is the symmetric bilinear form on $C_0^\infty(\Sg)$ given by
\begin{equation}
\label{eq:index1}
\indo_f(u,w):=\int_\Sg\left\{\escpr{\nabla_\Sg u,\nabla_\Sg w}-\big(\ric_f(N,N)+|\sigma|^{2}\big)\,u\,w\right\}da_{f}
-\int_{\ptl\Sg}\text{II}(N,N)\,u\,w\,dl_f.
\end{equation}
Here $\ric_f$ stands for the Bakry-\'Emery-Ricci tensor in \eqref{eq:fricci},  $|\sg|^2$ is the squared norm of the second fundamental form of $\Sg$, the symbol $\text{II}$ denotes the second fundamental form of $\ptl M$ with respect to the inner normal and $dl_f:=f\,dl$, where $dl$ is the $(n-1)$-dimensional Hausdorff measure in $(M,g)$. When $\ptl\Sg=\emptyset$ we adopt the convention that all the integrals along $\ptl\Sg$ vanish.

In the next examples we employ Lemma~\ref{lem:varprop} to characterize the stability in a particular situation, which includes horizontal half-spaces of Riemannian cylinders with product weights. Related results in Euclidean space were obtained by Barthe, Bianchini and Colesanti~\cite[Sect.~3]{chiara}, Doan~\cite[Thm.~5.1]{calibrations}, and Brock, Chiacchio and Mercaldo~\cite[Thm.~2.1]{bcm3}.

\begin{example}
\label{ex:stable}
Suppose that $E$ is a weighted stationary set such that $\Sg_0=\emptyset$ and $\Sg$ is a totally geodesic hypersurface with $\ric_f(N,N)=c$ on $\Sg$ and $\text{II}(N,N)=0$ along $\ptl\Sg$. Note that
\[
\indo_f(u,u)=\int_{\Sg}\big\{|\nabla_{\Sg}u|^2-c\,u^2\big\}\,da_f.
\]
Therefore $E$ is weighted stable if and only if $\la_f(\Sg)\geq c$, where $\la_f(\Sg)$ is the Poincar\'e constant of $\Sg$ with respect to the weight $f$ restricted to $\Sg$. 
\end{example}

\begin{remark} 
The \emph{Poincar\'e constant} of a complete weighted manifold $(M,g,f)$ is defined by
\begin{equation}
\label{eq:poincare}
\la_f(M):=\inf\left\{\frac{\int_M|\nabla u|^2\,dv_f}{\int_M u^2\,dv_f};\,u\in C^\infty_0(M),u\neq 0, \int_M u\,dv_f=0\right\}.
\end{equation}
An approximation argument shows that we can replace $C^\infty_0(M)$ with $H^1(M,dv_f)$ in the previous definition when $(M,g)$ is complete and $V_f(M)<\infty$. 
\end{remark}

\begin{example}
\label{ex:horstable}
Let $(M,g,f)$ be a Riemannian cylinder $\Om\times\rr$ with product weight $f(x,t):=e^{h(x)}\,e^{v(t)}$. It follows from \eqref{eq:fmc} that any horizontal slice $\Om_t:=\Om\times\{t\}$ has constant weighted mean curvature $H_f=v'(t)$ with respect to the unit normal $N=-\xi$. Thus, any lower horizontal half-space $\Om\times(-\infty,t)$ is a weighted stationary set. Note also that $\Om_t$ is totally geodesic, $\text{II}(\xi,\xi)=0$ along $\ptl\Om_t$ and $\ric_f(\xi,\xi)=-v''(t)$ on $\Om_t$ by equation~\eqref{eq:riccidesc}. From Example~\ref{ex:stable}, and using that  $\la_f(\Om_t)$ coincides with the Poincar\'e constant of $\Om$ with respect to $e^h$, we conclude that $\Om\times(-\infty,t)$ is weighted stable if and only if $\la_h(\Om)\geq -v''(t)$. The same criterion characterizes the stability of the upper horizontal half-space $\Om\times (t,\infty)$. 
\end{example}

Coming back to the weighted index form, an integration by parts formula~\cite[Eq.~(2.7)]{castro-rosales} yields 
\[
\mathcal{I}_f(u,w)=\mathcal{Q}_f(u,w), \ \text{ for any } u,w\in C^\infty_0(\Sg),
\]
where 
\begin{equation}
\label{eq:index2}
\mathcal{Q}_f(u,w):=-\int_\Sg u\,\mathcal{L}_f(w)\,da_f
-\int_{\ptl\Sg}u\,\left\{\frac{\ptl w}{\ptl\nu}+\text{II}(N,N)\,w\right\}dl_f.
\end{equation}
In the previous equation $\ptl w/\ptl\nu$ is the derivative of $w$ in the direction of $\nu$, and $\mathcal{L}_f$ is the \emph{weighted Jacobi operator of $\Sg$}, i.e., the second order linear operator defined by
\begin{equation}
\label{eq:jacobi}
\mathcal{L}_f(w):=\Delta_{\Sg,f}w+\left(\text{Ric}_f(N,N)+|\sg|^2\right)w.
\end{equation}
Here $\Delta_{\Sg,f}$ represents the \emph{weighted Laplacian} in $\Sg$. This is related to the Riemannian Laplacian $\Delta_\Sg$ by means of equality
\[
\Delta_{\Sg,f}w:=\Delta_\Sg w+\escpr{\nabla_\Sg\psi,\nabla_\Sg w}.
\]
For later use we recall the following equality, which shows a geometric interpretation of the weighted Jacobi operator, see \cite[Eq.~(3.5)]{castro-rosales} 
\begin{equation}
\label{eq:hfprima}
\big(\mathcal{L}_f (w)\big)(p)=\frac{d}{ds}\bigg|_{s=0}(H_f)_s(\phi_s(p)), 
\ \text{ for any } p\in\Sg.
\end{equation}
Here $(H_f)_s$ is the weighted mean curvature of the hypersurface $\Sg_s:=\phi_s(\Sg)$ for the variation of $E$ associated to a vector field $X$ tangent to $\ptl M$ and such that $\escpr{X,N}=w$ along $\Sg$. 

Observe that the symmetry property $\ind_f(u_1,u_2)=\ind_f(u_2,u_1)$ for $u_1,u_2\in C^\infty_0(\Sg)$ is equivalent to the identity
\begin{equation}
\label{eq:symmetry}
\int_\Sg \left\{u_1\,\mathcal{L}_f(u_2)-u_2\,\mathcal{L}_f(u_1)\right\}da_f
=\int_{\ptl\Sg}\left\{u_2\,\frac{\ptl u_1}{\ptl\nu}-u_1\,\frac{\ptl u_2}{\ptl\nu}\right\}dl_f.
\end{equation} 
When $\Sg$ is a weighted parabolic hypersurface, by taking a sequence $\{\var_k\}_{k\in\nn}$ as in \eqref{eq:sequence}, and using an approximation argument as in \cite[Lem.~2.4]{rosales-gauss} we see that equality \eqref{eq:symmetry} also holds for any bounded functions $u_1,u_2\in C^\infty(\Sg)\cap H^1(\Sg,da_f)$ such that $\Delta_{\Sg,f}u_i\in L^1(\Sg,da_f)$ and $\ptl u_i/\ptl\nu\in L^1(\ptl\Sg,dl_f)$, for any $i=1,2$. Here $L^1(\Sg,da_f)$, $L^1(\ptl\Sg,dl_f)$ and $H^1(\Sg,da_f)$ denote the corresponding spaces of integrable and Sobolev functions with respect to $da_f$ and $dl_f$. 

Motivated by Theorem~\ref{th:parabolicity} we are led to consider stable sets where the regular part of the interior boundary is a weighted parabolic hypersurface. In this situation the inequality in Lemma~\ref{lem:varprop} (ii) can be extended to more general functions that may not vanish on $\Sg_0$. This is done by assuming the conditions $\ric_f\geq 0$ and $\ptl M$ locally convex, under which the stability inequality becomes more restrictive since it contains two non-positive terms. The resulting statement is the following.

\begin{proposition}
\label{prop:extended}
Suppose that $\ptl M$ is locally convex and the weight $f$ satisfies $\emph{Ric}_f\geq 0$. If $E$ is a weighted stable set of finite weighted perimeter and such that $\Sg$ is weighted parabolic, then
\begin{itemize}
\item[(i)] For any bounded function $u\in H^1(\Sg,da_f)$ with $\int_\Sg u\,da_f=0$, we have $\big(\emph{Ric}_f(N,N)+|\sg|^2\big)\,u^2\in L^1(\Sg,da_f)$, $\emph{II}(N,N)\,u^2\in L^1(\ptl\Sg,dl_f)$ and
\[
\indo_f(u,u)\geq 0,
\]
where $\indo_f$ is the weighted index form in \eqref{eq:index1}.
\item[(ii)] For any bounded function $u\in C^\infty(\Sg)\cap H^1(\Sg,da_f)$ with $\Delta_{\Sg,f}u\in L^1(\Sg,da_f)$, $\ptl u/\ptl\nu\in L^1(\ptl\Sg,dl_f)$ and $\int_\Sg u\,da_f=0$, we have $(\emph{Ric}_f(N,N)+|\sg|^2)\,u^2\in L^1(\Sg,da_f)$, $\emph{II}(N,N)\,u^2\in L^1(\ptl\Sg,dl_f)$ and 
\[
\ind_f(u,u)\geq 0,
\]
where $\ind_f$ is defined in \eqref{eq:index2}. 
\end{itemize}
Moreover, in both cases $\emph{Ric}_f(N,N)+|\sg|^2\in L^1(\Sg,da_f)$ and $\emph{II}(N,N)\in L^1(\ptl\Sg,dl_f)$.
\end{proposition}

The proof of this proposition can be derived as in the Euclidean case~\cite[Prop.~4.2]{rosales-gauss} by means of an approximation argument that employs a sequence $\{\var_k\}_{k\in\nn}$ satisfying \eqref{eq:sequence}.

Next, we show some geometric and topological properties of stable sets that come immediately from Proposition~\ref{prop:extended} (i) by inserting constant functions in the weighted index form.

\begin{corollary}
\label{cor:connected}
Suppose that $\ptl M$ is locally convex and the weight $f$ satisfies $\emph{Ric}_f\geq 0$. If $E$ is a weighted stable set of finite weighted perimeter and $\Sg$ is weighted parabolic, then $\Sg$ is connected, or $\Sg$ is a totally geodesic hypersurface with $\emph{Ric}_f(N,N)=0$ on $\Sg$ and $\emph{II}(N,N)=0$ along $\ptl\Sg$.
\end{corollary}

As a consequence, we infer the connectivity of $\Sg$ when $\ric_f>0$. In the more general case $\ric_f\geq 0$ the hypersurface $\Sg$ need not be connected.

\begin{example}
\label{ex:disconnected}
Consider a Riemannian cylinder $\Om\times\rr$ with product weight $f(x,t):=e^{h(x)}$ such that $\ptl\Om$ is locally convex and $\ric_\Om\geq 0$. Then, according to Lemma~\ref{lem:varprop} (ii) and the computations in Example~\ref{ex:horstable}, any set $E$ for which the interior boundary is the union of finitely many horizontal slices is weighted stable. 
\end{example}

We finish this section with a criterion for weighted stability under our convexity assumptions.

\begin{example}
\label{ex:starig}
Suppose that $\ptl M$ is locally convex and $\ric_f\geq c\geq 0$. Consider a weighted stationary set $E$ in $M$ such that $\Sg_0=\emptyset$ and $\Sg$ is a connected totally geodesic hypersurface with $\ric_f(N,N)=c$ on $\Sg$ and $\text{II}(N,N)=0$ along $\ptl\Sg$. If $c=0$, then $E$ is weighted stable by Lemma~\ref{lem:varprop} (ii) and equation~\eqref{eq:index1}. Suppose that $c>0$ and $(M,g)$ has non-positive sectional curvature with respect to any plane containing the normal $N$ and a tangent direction to $\Sg$. In this situation, the restriction of $f$ to $\Sg$ provides a weight with Bakry-\'Emery-Ricci curvature $\ric^\Sg_f\geq c$. Moreover $\ptl\Sg$ is locally convex since $\ptl M$ is locally convex and $\Sg$ meets $\ptl M$ orthogonally along $\ptl\Sg$. Both conditions imply the estimate $\la_f(\Sg)\geq c$, see for instance \cite[Sect.~2.3]{kolesnikov-milman2} and the references therein. From Example~\ref{ex:stable} we conclude that $E$ is weighted stable.
\end{example}

\section{Main stability result}
\label{sec:stable}
\setcounter{equation}{0}

In this section we establish rigidity properties for stable sets in some weighted manifolds with locally convex boundary. This is done in Theorem~\ref{th:stable}, which extends to a more general setting the stability result in \cite[Thm.~4.6]{rosales-gauss} for Euclidean half-spaces or slabs endowed with certain perturbations of the Gaussian weight. For the proof we will use the inequality in Proposition~\ref{prop:extended} (ii) with a suitable function $u$. Motivated by the Euclidean situation, where we employed the test function $u:=\alpha+\escpr{\eta,N}$ for some $\alpha\in\rr$ and $\eta$ tangent to the boundary hyperplanes, in the present proof we will replace the constant vector $\eta$ with a unit parallel vector field $X$ tangent to $\ptl M$. 

Recall that a vector field $X$ on a Riemannian manifold $(M,g)$ is \emph{parallel} if $\nabla_UX=0$ for any vector field $U$. Here $\nabla$ stands for the Levi-Civita connection in $(M,g)$. It is clear that the Riemannian length $|X|$ of a parallel vector field $X$ is constant on $M$. Thus, the existence of non-vanishing parallel vector fields is equivalent to the existence of unit parallel vector fields. 

The next lemma gathers some computations involving parallel vector fields that will be useful in the proof of Theorem~\ref{th:stable}.

\begin{lemma}
\label{lem:eigen}
Let $(M,g,f)$ be a weighted manifold of dimension $n+1$. Suppose that $X$ is a parallel vector field on $M$ which is tangent to $\ptl M$ and has one-parameter flow $\{\phi_s\}_{s\in (-\eps,\eps)}$ for some $\eps>0$. For a hypersurface $\Sg\sub M$ with smooth unit normal $N$ we define the function $\var:=\escpr{X,N}$.
\begin{itemize}
\item[(i)] If $\Sg$ has constant weighted mean curvature, then
\[
\mathcal{L}_f(\var)=\emph{Ric}_f(X,N),
\]
where $\mathcal{L}_f$ is the weighted Jacobi operator defined in \eqref{eq:jacobi} and \emph{$\ric_f$} is the Bakry-\'Emery-Ricci tensor in \eqref{eq:fricci}. 
\vspace{0,1cm}
\item[(ii)] If $\Sg$ has boundary $\ptl\Sg=\Sg\cap\ptl M$ and meets $\ptl M$ orthogonally, then
\[
\frac{\ptl\var}{\ptl\nu}=-\emph{II}(N,N)\,\var,
\]
where $\nu$ is the conormal vector along $\ptl\Sg$ and $\emph{II}$ is the second fundamental form of $\ptl M$ with respect to the inner unit normal.
\end{itemize}
\end{lemma}

\begin{proof}
Since a parallel vector field is a Killing vector field \cite[p.~82]{dcriem}, the diffeomorphism $\phi_s$ is an isometry of $(M,g)$ for any $s\in (-\eps,\eps)$. Let $\Sg_s:=\phi_s(\Sg)$ and take the vector field $\overline{N}$, whose restriction to $\Sg_s$ is the unit normal $N_s$ defined by equality $N_s(\phi_s(p)):=(d\phi_s)_p(N(p))$ for any $p\in\Sg$. It follows that $[X,\overline{N}]=0$ along $\Sg$, see \cite[Prop.~5.4 in p.~28]{dcriem}. Denote by $(H_f)_s$ the weighted mean curvature of $\Sg_s$ with respect to $N_s$. Since the Riemannian mean curvature $H_s$ is preserved under isometries, we obtain by equation \eqref{eq:fmc} that
\[
(H_f)_s(\phi_s(p))=\big(nH_s-\escpr{\nabla\psi,N_s}\big)(\phi_s(p))=nH(p)-\escpr{\nabla\psi,N_s}(\phi_s(p)),
\]
for any $p\in\Sg$. Hence, the formula in \eqref{eq:hfprima} and the definition of $\ric_f$ in \eqref{eq:fricci} imply that
\begin{align*}
\mathcal{L}_f(\var)=\frac{d}{ds}\bigg|_{s=0}\big((H_f)_s\circ\phi_s\big)&=-(\nabla^2\psi)(X,N)-\escpr{\nabla\psi,\nabla_{X}\overline{N}}
\\
&=\ric_f(X,N)-\ric(X,N)-\escpr{\nabla\psi,\nabla_{\overline{N}}X}
\\
&=\ric_f(X,N),
\end{align*}
where we have used that $\ric(X,Y)=0$ and $\nabla_{Y}X=0$ for any vector field $Y$ because $X$ is parallel. 

Now we prove (ii). Denote by $\eta$ the inner unit normal on $\ptl M$. Since the hypersurface $\Sg$ meets $\ptl M$ orthogonally along $\ptl\Sg$, then $\escpr{\eta,N}=0$ and $\eta$ coincides with the conormal vector $\nu$ along $\ptl\Sg$. Thus, for any vector $w$ tangent to $\ptl\Sg$, we infer that $\escpr{\nabla_w\eta,N}+\escpr{\eta,\nabla_wN}=0$. As a consequence
\[
0=\text{II}(w,N)+\sg(w,\nu)=\text{II}(N,w)+\sg(\nu,w)=-\escpr{\nabla_N\eta+\nabla_\nu N,w}.
\]
This shows that $\nabla_N\eta+\nabla_\nu N=a\,\nu+b\,N$ along $\ptl\Sg$, where $b=-\text{II}(N,N)$. Since $X$ is parallel and tangent to $\ptl M$, we deduce
\begin{align*}
\frac{\ptl\var}{\ptl\nu}=\escpr{\nabla_\nu X,N}+\escpr{X,\nabla_\nu N}&=\escpr{X,a\,\nu}-\escpr{X,\nabla_N\eta}-\text{II}(N,N)\,\var
\\
\nonumber
&=\escpr{\nabla_NX,\eta}-\text{II}(N,N)\,\var=-\text{II}(N,N)\,\var,
\end{align*}
which completes the proof.
\end{proof}

\begin{remark}
From the equality in Lemma~\ref{lem:eigen} (i), if there is  $c\in\rr$ such that $\text{Ric}_f(X,Y)=c\,\escpr{X,Y}$ for any vector field $Y$, then $\mathcal{L}_f(\var)=c\,\var$, and so $\var$ is an eigenfunction for the weighted Jacobi operator. This holds in particular for gradient Ricci solitons, as was observed by Cheng, Mejia and Zhou~\cite[Cor.~4]{cmz2}, see also \cite[Prop.~1]{alencar-rocha}, after previous work of Colding and Minicozzi for the Gaussian solitons \cite[Lem.~5.5]{colding-minicozzi}, see also \cite[Lem.~5.1]{mcgonagle-ross}.
\end{remark}

Now, we are ready to prove the main result of this section.

\begin{theorem}
\label{th:stable}
Let $(M,g,f)$ be an $(n+1)$-dimensional weighted manifold, possibly with locally convex boundary. Suppose that $X$ is a unit parallel vector field on $M$ which is tangent to $\ptl M$ and has one-parameter flow $\{\phi_s\}_{s\in (-\eps,\eps)}$ for some $\eps>0$. Suppose also that there is $c>0$ such that
\begin{align*}
\emph{Ric}_f(X,Y)&=c\,\escpr{X,Y},\quad\text{for any vector field }\, Y,
\\
\emph{Ric}_f(Y,Y)&\geq c\,|Y|^2,\quad\text{for any vector field }\, Y\perp X.
\end{align*}
Consider an open set of finite weighted perimeter $E\sub M$ such that $\overline{\ptl E\cap\emph{int}(M)}=\Sg\cup\Sg_0$, where $\Sg$ is a smooth hypersurface with boundary $\ptl\Sg=\Sg\cap\ptl M$, and $\Sg_0$ is a closed singular set with $A_f(\Sg_0)=0$. If $E$ is weighted stable and $\Sg$ is weighted parabolic, then either $X$ is tangent to $\Sg$, or $\Sg$ is totally geodesic with $\emph{Ric}_f(N,N)=c$ on $\Sg$ and $\emph{II}(N,N)=0$ along $\ptl\Sg$.
\end{theorem}
 
\begin{proof}
From Lemma~\ref{lem:varprop} (i) we get that $\Sg$ has constant weighted mean curvature and meets $\ptl M$ orthogonally along $\ptl\Sg$. Note that $A_f(\Sg)<\infty$ since $E$ has finite weighted perimeter and $A_f(\Sg_0)=0$.

Take a unit vector field $Z$ on $M$. Write $Z=\rho X+Y$, where $\rho:=\escpr{X,Z}$ and $Y\perp X$. It is clear that $\rho^2+|Y|^2=1$. Thus, our first hypothesis about $\ric_f$ gives us
\begin{equation}
\label{eq:ricfcomp}
\begin{split}
\ric_f(Z,Z)&=\rho^2\,\ric_f(X,X)+\ric_f(Y,Y)+2\rho\,\ric_f(X,Y)=c\,\rho^2+\ric_f(Y,Y)
\\
&=c+\ric_f(Y,Y)-c\,|Y|^2.
\end{split}
\end{equation}
Our second hypothesis on $\ric_f$ yields $\ric_f\geq c$. From Corollary~\ref{cor:connected} we deduce that $\Sg$ is connected.

Consider the function $\var:=\escpr{X,N}$ on $\Sg$. Since $|\var|\leq 1$ on $\Sg$ and $A_f(\Sg)<\infty$, then $\var\in L^1(\Sg,da_f)$. We define the test function $u:=\alpha+\var$, where $\alpha:=-A_f(\Sg)^{-1}\int_\Sg\var\,da_f$. This is a bounded smooth function on $\Sg$ with $\int_\Sg u\,da_f=0$. Now we check that $u$ satisfies the integrability hypotheses in Proposition~\ref{prop:extended} (ii). From the fact that $X$ is parallel, we obtain
\begin{equation}
\label{eq:mika1}
(\nabla_\Sg u)(p)=(\nabla_\Sg\var)(p)=-\sum_{i=1}^nk_i(p)\,\escpr{X(p),e_i}\,e_i,
\end{equation} 
where $e_i$ is a principal direction of $\Sg$ at $p$ with associated principal curvature $k_i(p)$. Hence
\[
|\nabla_\Sg u|^2\leq |\sg|^2\leq \ric_f(N,N)+|\sg|^2,
\] 
which is an integrable function with respect to $da_f$ by Proposition~\ref{prop:extended}. From this we infer that $u\in H^1(\Sg,da_f)$. On the other hand, the weighted Jacobi operator of $\Sg$ satisfies the equality
\[
 \mathcal{L}_f(\var)=\ric_f(X,N)=c\,\var
\] 
by Lemma~\ref{lem:eigen} (i) and the first hypothesis about $\ric_f$. Thus, we have
\[
\Delta_{\Sg,f}\,u=\Delta_{\Sg,f}\,\var=c\,\var-(\ric_f(N,N)+|\sg|^2)\,\var,
\]
and so $\Delta_{\Sg,f}u\in L^1(\Sg,da_f)$. As to the boundary term, the equality in Lemma~\ref{lem:eigen} (ii) gives us
\begin{equation}
\label{eq:bdterm}
\frac{\ptl u}{\ptl\nu}=\frac{\ptl\var}{\ptl\nu}=-\text{II}(N,N)\,\var,
\end{equation}
and we deduce that $\ptl u/\ptl\nu\in L^1(\ptl\Sg,dl_f)$ from the conclusion $\text{II}(N,N)\in L^1(\ptl\Sg,dl_f)$ in the statement of Proposition~\ref{prop:extended}.

At this point, the weighted stability of $E$ implies by Proposition~\ref{prop:extended} (ii) that $\ind_f(u,u)\geq 0$. Next, we shall compute the different terms in $\ind_f(u,u)$. Denote $Y:=N-\var\,X$. By equality \eqref{eq:ricfcomp} we get $\ric_f(N,N)=c+\ric_f(Y,Y)-c\,|Y|^2$ on $\Sg$. We also know that $\mathcal{L}_f(\var)=c\,\var$. As a consequence
\begin{equation}
\label{eq:memo}
\begin{split}
\int_\Sg u\,\mathcal{L}_f (u)\,da_f&=\int_\Sg (\alpha+\var)\,\mathcal{L}_f(\alpha+\var)\,da_f
\\
&=\int_\Sg (\alpha+\var)\left\{\alpha\,\big(c+\ric_f(Y,Y)-c\,|Y|^2+|\sg|^2\big)+c\,\var\right\}da_f
\\
&=\int_\Sg (\alpha+\var)\left\{\alpha\,\big(\ric_f(Y,Y)-c\,|Y|^2+|\sg|^2\big)+c\,\var\right\}da_f,
\end{split}
\end{equation}
where in the last equality we have used that $\int_\Sg u\,da_f=0$. By taking into account equations \eqref{eq:symmetry} and \eqref{eq:bdterm}, we obtain
\begin{align*}
\int_\Sg c\,\var\,da_f&=\int_\Sg\mathcal{L}_f(\var)\,da_f
=\int_\Sg \var\,\mathcal{L}_f(1)\,da_f-\int_{\ptl\Sg}\frac{\ptl \var}{\ptl\nu}\,dl_f
\\
&=\int_\Sg \var\,\big(c+\ric_f(Y,Y)-c\,|Y|^2+|\sg|^2\big)\,da_f+\int_{\ptl\Sg}\text{II}(N,N)\,\var\,dl_f,
\end{align*}
so that
\[ 
\int_\Sg \var\,\big(\ric_f(Y,Y)-c\,|Y|^2+|\sg|^2\big)\,da_f=-\int_{\ptl\Sg}\text{II}(N,N)\,\var\,dl_f. 
\]
Plugging this into \eqref{eq:memo}, and having in mind that $\alpha=-A_f(\Sg)^{-1}\int_\Sg \var\,da_f$, it follows that
\begin{equation}
\label{eq:premio1}
\begin{split}
\int_\Sg u\,\mathcal{L}_f(u)\,da_f&=\alpha^2\int_\Sg\big(\ric_f(Y,Y)-c\,|Y|^2+|\sg|^2\big)\,da_f
\\
&+c\,\left[\int_\Sg \var^2\,da_f-\frac{1}{A_f(\Sg)}\left(\int_\Sg \var\,da_f\right)^2\right]-\alpha\int_{\ptl\Sg}\text{II}(N,N)\,\var\,dl_f.
\end{split}
\end{equation}
On the other hand, equation \eqref{eq:bdterm} yields
\begin{align}
\label{eq:premio2}
\int_{\ptl\Sg}u\left\{\frac{\ptl u}{\ptl\nu}+\text{II}(N,N)\,u\right\}dl_f=\alpha\int_{\ptl\Sg}\text{II}(N,N)\,u\,dl_f.
\end{align}
Therefore, equalities \eqref{eq:premio1} and \eqref{eq:premio2} together with the stability inequality, lead to
\begin{equation}
\label{eq:opera}
\begin{split}
0\leq\ind_f(u,u)=&-\alpha^2\int_\Sg\big(\ric_f(Y,Y)-c\,|Y|^2+|\sg|^2\big)\,da_f
\\
&-c\left[\int_\Sg \var^2\,da_f-\frac{1}{A_f(\Sg)}\left(\int_\Sg \var\,da_f\right)^2\right]
\\
&-\alpha^2\int_{\ptl\Sg}\text{II}(N,N)\,dl_f\leq 0,
\end{split}
\end{equation}
where we have used that $\ric_f(Y,Y)\geq c\,|Y|^2$, the Cauchy-Schwarz inequality in $L^2(\Sg,da_f)$ and the local convexity of $\ptl M$. Hence, the three terms at the right hand side of \eqref{eq:opera} vanish. In particular $\var(p)=-\alpha$ for any $p\in\Sg$. If $\alpha=0$, then $X$ is tangent to $\Sg$. Otherwise we get $|\sg|^2=0$ (i.e. $\Sg$ is totally geodesic), $\text{II}(N,N)=0$ along $\ptl\Sg$ and $\ric_f(Y,Y)=c\,|Y|^2$, which by \eqref{eq:ricfcomp} implies that $\ric_f(N,N)=c$ on $\Sg$. 
This completes the proof.
\end{proof}

\begin{remarks}
\label{re:muchas}
(i). By assuming that $c=0$ (instead of $c>0$) and that $\Sg$ is disconnected or satisfies $\int_\Sg\escpr{X,N}\,da_f\neq 0$, we conclude that $\Sg$ is totally geodesic with $\ric_f(N,N)=\text{II}(N,N)=0$. When $\Sg$ is disconnected it suffices to apply Corollary~\ref{cor:connected}. Otherwise the proof goes as in the case $c>0$ with the only difference that the second summand in equation~\eqref{eq:opera} vanishes. 

(ii). For a given stable set $E$ with $\ptl\Sg\neq\emptyset$, the theorem is true by assuming the local convexity condition and the tangency property of $X$ only in a neighborhood of $\ptl\Sg$ in $\ptl M$. 

(iii). From Theorem~\ref{th:regularity} the singular set $\Sg_0$ in the interior boundary of a weighted minimizer consists of points where $|\sg|^2$ blows up. For a weighted stable set $E$ with a singular set $\Sg_0$ of this type the conclusion that $\Sg$ is totally geodesic implies $\Sg_0=\emptyset$ and $\Sg=\overline{\ptl E\cap\text{int}(M)}$. On the other hand, from the fact that $X$ is tangent to $\Sg$ it follows that $\Sg$ carries a unit vector field which is parallel with respect to the induced metric. As we see ``below'' this entails some restrictions on the topology and geometry of $\Sg$.

(iv). By reasoning as in \cite[Lem.~4.5]{rosales-gauss} we can give a geometric interpretation of the test function $u:=\alpha+\escpr{X,N}$. For a bounded set $E$ with smooth interior boundary, this function is associated with a volume-preserving variation obtained by combining the one-parameter flow of $X$ with the one-parameter flow of a vector field $\overline{N}$ on $M$, which is tangent to $\ptl M$, and satisfies $\overline{N}=N$ on $\Sg$.

(v). Recall that a hypersurface $\Sg$ immersed in $M$ is $f$-stable if it is a second order minimum of the interior weighted area under variations preserving the weighted volume, see for instance \cite[Sect.~3]{castro-rosales}. It is straightforward to check that the rigidity properties in Theorem~\ref{th:stable} are also valid for any orientable, weighted parabolic and $f$-stable hypersurface $\Sg$ with $A_f(\Sg)<\infty$. In the particular case $\ptl\Sg=\emptyset$, this generalizes previous results of McGonagle and Ross~\cite[Cor.~4.8]{mcgonagle-ross} in Gauss space, and of Alencar and Rocha~\cite[Thm.~2]{alencar-rocha} for arbitrary shrinkers. 
\end{remarks}

In order to show the scope of Theorem~\ref{th:stable} we analyze the different hypotheses in the statement and list some interesting situations where the result applies.

\emph{1. Manifolds with parallel vector fields}. If $X$ is a unit parallel vector field on a Riemannian manifold $(M,g)$, then the sectional curvature vanishes for any tangent plane containing $X$. Hence, we have the curvature condition $\ric(X,Y)=0$ for any vector field $Y$ and so, the hypothesis $\ric_f(X,Y)=c\,\escpr{X,Y}$ is equivalent to that $(\nabla^2\psi)(X,Y)=-c\,\escpr{X,Y}$ for any $Y$. Other geometric and topological restrictions are found in two papers of Welsh, see~\cite[Thms.~1 and 3]{welsh1}, \cite[Prop.~2.1]{welsh2} and the references therein. As a consequence, a complete Riemannian manifold admits a unit parallel vector field if and only if it is isometric to a quotient $(\Om\times\rr)/G$, where $\Om\times\rr$ is a Riemannian cylinder and $G$ is a subgroup of $\text{Iso}(\Om)\times\rr$ such that the projection $G\to\rr$ is injective and the orbits of $G$ in $\Om\times\rr$ are discrete sets. 

\emph{2. Riemannian cylinders}. The previous observation leads to consider the simplest example of a Riemannian cylinder $M:=\Om\times\rr$ with its vertical parallel vector field $\xi$. We know that $\ptl M$ is locally convex if and only if $\ptl\Om$ is locally convex. For a product weight $f(x,t):=e^{h(x)}\,e^{v(t)}$ we infer from equation \eqref{eq:riccidesc} that 
\begin{align*}
\ric_f(\xi,Y)&=-(v''\circ\pi)\,\escpr{\xi,Y}, \quad\text{for any vector field } Y,
\\
\ric_f(Y,Y)&=\ric_h(Y_*,Y_*), \quad\text{for any vector field } Y\perp\xi,
\end{align*}
where $\pi(x,t):=t$ and $Y_*(x,t)\in T_x\Om$ is the horizontal projection of $Y(x,t)$. So, the hypotheses involving $\ric_f$ in Theorem~\ref{th:stable} are equivalent to that the horizontal weight $e^h$ satisfies $\ric_h\geq c$ on $\Om$, and $e^v$ is an affine perturbation of the $c\,$-Gaussian soliton in $\rr$, i.e., $v(t)=-c\,t^2/2-bt-a$ for some $a,b\in\rr$. In Section~\ref{sec:cylinders} we will discuss in more detail the classification of stable sets and the uniqueness of weighted minimizers in this setting.

\emph{3. Gradient Ricci solitons}. Our assumptions on $\ric_f$ are satisfied when $(M,g,f)$ is a shrinker or a steady soliton (see Remarks~\ref{re:muchas} (i)). By a result of Petersen and Wylie~\cite[Cor.~2.2]{petersen-wylie}, the existence in a gradient Ricci soliton $(M,g,f)$ of a unit parallel vector field $X$ entails that $(M,g)$ is isometric to a Riemannian cylinder $\Om\times\rr$, or $\escpr{\nabla\psi,X}=0$ on $M$. In the second case, we have $(\nabla^2\psi)(X,X)=0$, so that $c=\ric_f(X,X)=0$. Thus, the shrinkers for which Theorem~\ref{th:stable} holds are Riemannian cylinders $\Om\times\rr$ with a product weight $f(x,t):=e^{h(x)}\,e^{v(t)}$ such that $\Om$ is a $c\,$-gradient soliton with respect to $e^h$, and $e^v$ is an affine perturbation of the $c\,$-Gaussian soliton. The case of steady solitons is less rigid, as there are examples not isometric to a Riemannian cylinder (like unweighted compact Ricci-flat manifolds having a unit parallel vector field). 

\emph{4. Horizontal perturbations}. The hypothesis $\ric_f(Y,Y)\geq c\,|Y|^2$ for any vector field $Y\perp X$ allows to apply Theorem~\ref{th:stable} for more general weighted manifolds than the gradient Ricci solitons described above. A related situation appears when we consider the Riemannian cylinder $M:=\Om\times\rr$ with a product weight $f(x,t):=e^{\mu(x)+\omega(x)}\,e^{-ct^2/2}$, where $e^\mu$ is a weight with $\ric_\mu=c$ on $\Om$ and $\omega$ is any smooth concave function on $\Om$.

\emph{5. Integral condition.} When $c=0$ and $\Sg$ is connected we need that $\int_\Sg\escpr{X,N}\,da_f\neq 0$ to deduce the rigidity properties of stable sets. Note that, for a bounded set $E$ with smooth interior boundary $\Sg$, the divergence theorem and equation~\eqref{eq:divf} lead to
\begin{equation}
\label{eq:divthweight}
\int_\Sg\escpr{X,N}\,da_f=-\int_{E}\divv_f X\,dv_f=-\int_E\escpr{\nabla\psi,X}\,dv_f,
\end{equation}
because $X$ is parallel and tangent to $\ptl M$. This guarantees that $\int_\Sg\escpr{X,N}\,da_f\neq 0$ if the term $\escpr{\nabla\psi,X}$ never vanishes on $M$. On the contrary, in the unweighted case we have $\int_\Sg\escpr{X,N}\,da=0$. In Section~\ref{subsec:zerocase} we will analyze this condition and its consequences for weighted cylinders.

\section{Stability and isoperimetry in weighted cylinders}
\label{sec:cylinders}

In this section we consider Riemannian cylinders with locally convex boundary and a weight $f$ such that $\ric_f\geq c>0$. In Section~\ref{subsec:unique} we provide conditions that allow to classify stable sets and to deduce that the horizontal half-spaces are the only minimizers. In Section~\ref{subsec:examples} we show some situations where our results apply. Finally, in Section~\ref{subsec:zerocase} we briefly discuss the case $c=0$.

\subsection{Uniqueness results}
\label{subsec:unique}
\noindent

We begin with a statement that improves Theorem~\ref{th:stable} by showing that the rigidity properties for a weighted stable set still hold when $\xi$ is tangent to the regular part of its interior boundary. 

\begin{theorem}
\label{th:stacyl}
Let $(M,g,f)$ be a weighted cylinder with base space $\Om$, possibly with locally convex boundary, and weight of the form
\[
f(x,t):=e^{h(x)}\,e^{v(t)},
\]
where the function $h\in C^\infty(\Om)$ satisfies $\emph{Ric}_h\geq c>0$, and $e^v$ is an affine perturbation of the $c$-Gaussian soliton on $\rr$. Consider an open set of finite weighted perimeter $E\sub M$ such that $\overline{\ptl E\cap\emph{int}(M)}=\Sg\cup\Sg_0$, where $\Sg$ is a smooth hypersurface with boundary $\ptl\Sg=\Sg\cap\ptl M$, and $\Sg_0$ is a closed singular set with $A_f(\Sg_0)=0$. If $E$ is weighted stable and $\Sg$ is weighted parabolic, then $\Sg$ is a connected totally geodesic hypersurface with $\emph{Ric}_f(N,N)=c$ on $\Sg$ and $\emph{II}(N,N)=0$ along $\ptl\Sg$.
\end{theorem}

\begin{proof}
The connectivity of $\Sg$ comes from Corollary~\ref{cor:connected}. By Theorem~\ref{th:stable} it is enough to see that, if $\Sg$ is a vertical hypersurface, then $\Sg$ is totally geodesic with $\ric_f(N,N)=c$ on $\Sg$ and $\text{II}(N,N)=0$ along $\ptl\Sg$. The fact that $\Sg$ is vertical implies that it is foliated by integral curves of $\xi$, i.e., by vertical segments. Since a point in $\Sg_0$ cannot appear along these segments we get $\Sg=\Sg_*\times\rr$ for some smooth hypersurface $\Sg_*$ in $\Om$ with $\ptl\Sg_*=\Sg_*\cap\ptl\Om$. The weighted area $A_h(\Sg_*)$ of $\Sg_*$ with respect to $e^h$ is finite because it coincides, up to a multiplicative constant, with $A_f(\Sg)=P_f(E)<\infty$. 

Let $u:\Sg\to\rr$ be a function defined by $u(x,t):=\psi(t)$, where $\psi\in C^\infty_0(\rr)$ with $\psi\neq 0$ and $\int_\rr\psi\,dl_v=0$. Here $dl_v$ is the weighted length measure in $\rr$ for the weight $e^v$. By having in mind the stability inequality in Proposition~\ref{prop:extended} (i), Fubini's theorem, and that the functions $q_1:=\ric_f(N,N)+|\sg|^2$ and $q_2:=\text{II}(N,N)$ only depend on the horizontal component of the points in $\Sg$ and $\ptl\Sg$, we obtain
\[
0\leq\indo_f(u,u)=A_h(\Sg_*)\,\int_\rr(\psi')^2\,dl_v-\left(\int_{\Sg_*} q_1\,da_h+\int_{\ptl\Sg_*}q_2\,dl_h\right)\int_\rr\psi^2\,dl_v,
\]
where $da_h$ and $dl_h$ denote the corresponding horizontal weighted measures in $\Om$. It follows that
\[
\frac{\int_\rr(\psi')^2\,dl_v}{\int_\rr\psi^2\,dl_v}\geq\frac{\int_{\Sg_*}q_1\,da_h+\int_{\ptl\Sg_*}q_2\,dl_h}{A_h(\Sg_*)}\geq c,
\]
where we have used that $\ric_f(N,N)\geq c$ on $\Sg$, $|\sg|^2\geq 0$ on $\Sg$ and $\text{II}(N,N)\geq 0$ along $\ptl\Sg$. The proof finishes by taking into account that the Poincar\'e constant defined in \eqref{eq:poincare} of $\rr$ with respect to the weight $e^v$ equals $c$, see for instance Ledoux~\cite[Sect.~2.1]{markov}.
\end{proof}

Under the conditions of Theorem~\ref{th:stacyl} any horizontal half-space is a weighted stable set, see Example~\ref{ex:starig} or Example~\ref{ex:horstable}. In general, other weighted stable sets can appear. For instance, if $M$ is a Euclidean half-space or slab in a $c\,$-Gaussian soliton with $c>0$, then the intersection with $M$ of any Euclidean half-space with boundary perpendicular o parallel to $\ptl M$ is weighted stable. The latter case shows an example where $\ptl\Sg=\emptyset$ even though $\ptl M\neq\emptyset$.

In the next result we provide geometric and analytic conditions on the base manifold $\Om$ and the weight $e^h$ ensuring that the horizontal half-spaces are the only weighted stable sets.

\begin{corollary}
\label{cor:staunique}
Let $(M,g,f)$ be a weighted cylinder with base space $\Om$, possibly with locally convex boundary, and weight of the form
\[
f(x,t):=e^{h(x)}\,e^{v(t)},
\]
where the function $h\in C^\infty(\Om)$ satisfies $\emph{Ric}_h\geq c>0$, and $e^v$ is an affine perturbation of the $c$-Gaussian soliton on $\rr$. Consider an open set of finite weighted perimeter $E\sub M$ such that $\overline{\ptl E\cap\emph{int}(M)}=\Sg\cup\Sg_0$, where $\Sg$ is a smooth hypersurface with boundary $\ptl\Sg=\Sg\cap\ptl M$, and $\Sg_0$ is a closed singular set with $A_f(\Sg_0)=0$. Suppose that $E$ is weighted stable, $\Sg$ is weighted parabolic, and some of the following conditions hold:
\begin{itemize}
\item[(i)] $\emph{Ric}_h>c$, or
\item[(ii)] $\ptl\Om$ is locally strictly convex and $\ptl\Sg\neq\emptyset$.
\end{itemize}
Then, $E$ is a horizontal half-space.
\end{corollary}

\begin{proof}
We know by Theorem~\ref{th:stacyl} that $\Sg$ is a connected totally geodesic hypersurface satisfying $\ric_f(N,N)=c$ on $\Sg$ and $\text{II}(N,N)=0$ along $\ptl\Sg$. We write $N=Y+\var\,\xi$, where $\var:=\escpr{\xi,N}$. Note that $\var$ is constant on $\Sg$ because $\Sg$ is totally geodesic and $\xi$ is a parallel vector field.

The hypothesis $\ric_h>c$ means that $(\ric_h)_x(w_*,w_*)>c\,|w|^2$ for any $x\in\Om$ and $w_*\in T_x\Om$ with $w_*\neq 0$. By taking $Z=N$ in equation \eqref{eq:ricfcomp} we get $\ric_f(Y,Y)=c\,|Y|^2$. From equation \eqref{eq:riccidesc} we deduce $\ric_h(Y_*,Y_*)=c\,|Y_*|^2$. Hence $Y_*=Y=0$, so that $N=\xi$ or $N=-\xi$ on $\Sg$. Then, the differential of the projection $\pi:\Sg\to\rr$ given by $\pi(x,t):=t$ vanishes everywhere. Thus, there is a horizontal slice $\Om_t:=\Om\times\{t\}$ such that $\Sg\subseteq\Om_t$. Since $\Om_t$ is a closed hypersurface in $M$ it follows that $\overline{\Sg}\subseteq\Om_t$. This entails $\Sg_0\cap\overline{\Sg}=\emptyset$ and $\Sg=\Om_t$, so that $E$ is a horizontal half-space. 

Suppose now that $\ptl\Om$ is locally strictly convex and $\ptl\Sg\neq\emptyset$. From the orthogonality condition between $\Sg$ and $\ptl M$ we have that $Y$ is tangent to $\ptl M$ along $\Sg$. Then, the equality $0=\text{II}(N,N)=\text{II}(Y,Y)$ leads to $Y=0$, and so $N=\pm\xi$ along $\ptl\Sg$. Since $\var$ is constant on $\Sg$ we conclude that $N=\xi$ or $N=-\xi$ on $\Sg$. From here we finish as in the previous case.
\end{proof}

As we noticed in the previous proof, the fact that $\Sg$ is totally geodesic entails that $\escpr{\xi,N}=\theta$ on $\Sg$ for some $\theta\in[-1,1]$. The case $\theta^2=1$ implies that $\Sg$ is a horizontal slice whereas $\theta=0$ means that $\Sg$ is vertical. Half-spaces orthogonal to a Gaussian slab provide examples where $0<\theta^2<1$. The next characterization result shows a situation where the case $0<\theta^2<1$ is not possible. 

\begin{corollary}
\label{cor:stasol}
Let $\Om$ be a Riemannian manifold, possibly with locally convex boundary, and satisfying $\emph{\ric}_\Om\geq c>0$. Consider the Riemannian cylinder $M:=\Om\times\rr$ with a vertical weight $f(x,t):=e^{v(t)}$ that is an affine perturbation of the $c$-Gaussian soliton on $\rr$. Take an open set of finite weighted perimeter $E\sub M$ such that $\overline{\ptl E\cap\emph{int}(M)}=\Sg\cup\Sg_0$, where $\Sg$ is a smooth hypersurface with boundary $\ptl\Sg=\Sg\cap\ptl M$, and $\Sg_0$ is a closed singular set with $A_f(\Sg_0)=0$. If $E$ is weighted stable and $\Sg$ is weighted parabolic then, either $E$ is a horizontal half-space, or $\Sg$ is a vertical, connected, totally geodesic hypersurface in $M$ such that $\emph{Ric}_f(N,N)=c$ on $\Sg$ and $\emph{II}(N,N)=0$ along $\ptl\Sg$.
\end{corollary}

\begin{proof}
From Theorem~\ref{th:stacyl} we know that $\Sg$ is a connected totally geodesic hypersurface in $M$ with $\ric_f(N,N)=c$ on $\Sg$ and $\text{II}(N,N)=0$ along $\ptl\Sg$. By Lemma~\ref{lem:varprop} (i) the weighted mean curvature $H_f$ is constant on $\Sg$. From equation~\eqref{eq:fmc} and the fact that $\nabla\psi=-(c\,\pi+b)\,\xi$ on $M$ for some $b\in\rr$, we obtain $H_f=(c\,\pi+b)\,\var$ on $\Sg$, where $\var:=\escpr{\xi,N}$ is a constant function with $|\var|\leq 1$. If we have $\var\neq 0$, then $\Sg$ is contained in a horizontal slice $\Om_t$ and we conclude that $E$ is a horizontal half-space. If $\var=0$, then $\xi$ is tangent to $\Sg$. This completes the proof.
\end{proof}

Now, we turn to the analysis of weighted isoperimetric regions. Let $(M,g,f)$ be a weighted cylinder in the hypotheses of Theorem~\ref{th:stacyl}. By a result of Morgan~\cite[Cor.~4]{morgandensity} the fact that $\ric_h\geq c>0$ entails $V_h(\Om)<\infty$ and so, any horizontal slice $\Om_t:=\Om\times\{t\}$ has finite weighted area. Our aim is to seek conditions ensuring that the regions bounded by these slices uniquely minimize the weighted perimeter for fixed weighed volume. 

The isoperimetric property of horizontal half-spaces can be proved in several ways, see for instance Theorem~\ref{th:another}. In general, we can find weighted isoperimetric regions different from horizontal half-spaces. Indeed, if a weighted manifold admits several structures of weighted cylinder in the conditions of Theorem~\ref{th:stacyl}, then any ``horizontal half-space'' associated to one of these structures is a minimizer. This happens for instance in a Gaussian slab, where all half-spaces meeting the boundary orthogonally minimize. 

In the next corollary we deduce from the stability condition rigidity properties of arbitrary minimizers and uniqueness results. A different proof mainly based on the second variation formula is given in Theorem~\ref{th:another}.

\begin{corollary}
\label{cor:isocyl}
Let $(M,g,f)$ be a weighted cylinder with base space $\Om$, possibly with locally convex boundary, and weight of the form
\[
f(x,t):=e^{h(x)}\,e^{v(t)},
\]
where $e^h$ is a smooth bounded function on $\Om$ with $\emph{Ric}_h\geq c>0$, and $e^v$ is an affine perturbation of the $c$-Gaussian soliton on $\rr$. If $E$ is a weighted minimizer with regular and singular parts of its interior boundary denoted by $\Sg$ and $\Sg_0$, respectively, then $\Sg$ is a connected totally geodesic hypersurface, $\Sg_0=\emptyset$, $\emph{Ric}_f(N,N)=c$ on $\Sg$ and $\emph{II}(N,N)=0$ along $\ptl\Sg$. Moreover, if some of the following conditions hold: 
\begin{itemize}
\item[(i)] $\emph{Ric}_h>c$, or
\item[(ii)] $\ptl\Om$ is locally strictly convex and $\ptl\Sg\neq\emptyset$,
\end{itemize}
then $E$ is a horizontal half-space, up to a set of volume zero.
\end{corollary}

\begin{proof}
Let $E$ be a weighted isoperimetric region. From Theorems~\ref{th:regularity} and \ref{th:parabolicity} we know that $E$ is a stable set of finite weighted perimeter, $\Sg$ is a weighted parabolic hypersurface possibly with boundary $\ptl\Sg=\Sg\cap\ptl M$, and  $A_f(\Sg_0)=0$. Hence, we can employ Theorem~\ref{th:stacyl} and Corollary~\ref{cor:staunique} to infer the claim.  The conclusion $\Sg_0=\emptyset$ comes from Theorem~\ref{th:regularity} since $\Sg$ is totally geodesic.  
\end{proof}

\subsection{Examples}
\label{subsec:examples}
\noindent

We now discuss interesting situations where our results are applied. In some of the cases where different stable candidates appear we will employ comparison arguments to characterize the weighted isoperimetric regions. 

\begin{example}[Cylindrical shrinkers]
\label{ex:main1}
Consider a Riemannian cylinder $P\times\rr$ over a complete manifold $P$ with $\ptl P=\emptyset$. At the end of Section~\ref{subsec:cylinders} we pointed out that a weight $f$ on $P\times\rr$ produces a shrinker with $\ric_f=c$ if and only if $f(x,t)=e^{h(x)}\,e^{v(t)}$, where $\ric_h=c$ and $e^v$ is an affine perturbation of the $c\,$-Gaussian soliton on $\rr$. When $P$ is noncompact it was proved by Cao and Zhou~\cite[Thm.~1.1]{cao-zhou} that there is a point $x_0\in P$ and constants $\alpha,\beta>0$ such that
\[
h(x)\leq-\alpha\,(d(x)-\beta)^2,\quad\text{for any } x\in P,
\]
where $d(x)$ is the distance function in $P$ from $x_0$. This implies that $e^h$ is bounded. Our results provide rigidity of weighted stable sets and minimizers in any cylinder $M:=\Om\times\rr$, where $\Om\subseteq P$ has locally convex boundary. Moreover, if $\ptl\Om$ is locally strictly convex, then any weighted stable set $E$ with $\ptl\Sg\neq\emptyset$ is a horizontal half-space. 
\end{example}

\begin{example}[Cylinders with Gaussian vertical weights]
\label{ex:cyleinstein}
Let $\Om$ be a complete manifold such that $\ptl\Om=\emptyset$ and $\ric_\Om\geq c>0$. Consider the Riemannian cylinder $M:=\Om\times\rr$ with weight $f(x,t):=e^{-ct^2/2}$. A weighted stable set as in Corollary~\ref{cor:stasol} must be a horizontal half-space, or the regular part of its interior boundary is a vertical, connected, totally geodesic hypersurface in $M$. Let us see that only the horizontal half-spaces minimize. 

Let $E$ be a weighted isoperimetric region. From Corollary~\ref{cor:isocyl} and the classification of weighted stable sets above, either $E$ is a horizontal half-space (up to a set of volume zero), or the interior boundary $\Sg$ equals $\Sg_*\times\rr$ for some connected and totally geodesic hypersurface $\Sg_*\sub\Om$. Since $\Om$ is a compact manifold (because $\ric_\Om\geq c>0$) and $\Sg_*$ is a closed subset in $\Om$ then $\Sg_*$ is compact. To prove our claim it suffices to check that
\[
A_f(\Sg_*\times\rr)>A_f(\Om_t),
\]
where $\Om_t:=\Om\times\{t\}$ is any horizontal slice and $\Sg_*$ is any compact and connected minimal hypersurface in $\Om$. This inequality is equivalent to
\[
A(\Sg_*)>\sqrt\frac{c}{2\pi}\,V(\Om)\,e^{-ct^2/2},
\]
where $A(\Sg_*)$ and $V(\Om)$ denote the Riemannian area and volume of $\Sg_*$ and $\Om$, respectively. Thanks to an estimate of  Maeda~\cite[Thm.~C]{maeda} we obtain
\[
A(\Sg_*)\geq\frac{V(\Om)}{2}\,\bigg(\int_0^{d(\Om)}e^{-ct^2/2}\,dt\bigg)^{-1},
\]
where $d(\Om)$ stands for the Riemannian diameter of $\Om$. From here the desired comparison follows by using that $\int_0^\infty e^{-ct^2/2}\,dt=\sqrt{\pi/(2c)}$ and $e^{-ct^2/2}\leq 1$.

A remarkable case of this example occurs when $\Om$ is an Einstein manifold of positive Ricci curvature, like the sphere $\mathbb{S}^n$, the real projective space $\mathbb{P}^n$ or a lens space.
\end{example}

\begin{example}[Horizontal perturbations]
\label{ex:main2}
Let $M:=\Om\times\rr$ be a Riemannian cylinder with locally convex boundary, and suppose that $M$ is a shrinker with respect to a weight $e^{\mu(x)}\,e^{v(t)}$. For a smooth concave function $\omega$ on $\Om$, we define a horizontal perturbation of the shrinker by $f(x,t):=e^{\mu(x)+\omega(x)}\,e^{v(t)}$. By concavity, this new weight satisfies $\ric_h\geq\ric_\mu=c$, and so our results hold in this setting. Furthermore, if $\omega$ is strictly concave then $\ric_h>c$, which entails uniqueness of horizontal half-spaces as weighted stable sets and weighted minimizers.
\end{example}

\begin{example}[Euclidean weighted cylinders]
\label{ex:main3}
Let $\Om\subseteq\rr^n$ be the closure of a smooth convex domain. Consider the Euclidean solid cylinder $M:=\Om\times\rr$ with a product weight $f(x,t):=e^{h(x)}\,e^{v(t)}$ such that $\ric_h\geq c>0$ and $e^v$ is an affine perturbation of the $c\,$-Gaussian soliton. By equation~\eqref{eq:fricci} and \cite[Cor.~4]{morgandensity} the function $e^h$ is log-concave and $V_h(\Om)<\infty$. Thus, we can reason as in \cite[Lem.~10.6.1]{agm} to infer that $e^h$ is bounded on $\Om$. 

For a weighted stable set $E$ in the conditions of Theorem~\ref{th:stacyl} the interior boundary is a closed, connected, totally geodesic hypersurface $\Sg$ in $\rrn$ with $\ric_f(N,N)=c$. From here we get that $E$ is the intersection with $M$ of a Euclidean half-space.  Moreover, if $\ptl\Sg\neq\emptyset$, then $\Sg$ meets $\ptl M$ orthogonally and $\text{II}(N,N)=0$ along $\ptl\Sg$. The converse is true by Example~\ref{ex:starig}: if a half-space $E$ intersected with $M$ is weighted stationary and satisfies $\ric_f(N,N)=c$ on $\Sg$ and $\text{II}(N,N)=0$ along $\ptl\Sg$, then it is weighted stable.

If we further assume that $\ric_h>c$, then Corollaries~\ref{cor:staunique} and \ref{cor:isocyl} yield that the horizontal half-spaces are the unique weighted stable and isoperimetric regions. If we suppose that $\Om$ is strictly convex, then a hyperplane cannot be entirely contained in $\text{int}(M)$. As a consequence $\ptl\Sg\neq\emptyset$ and we deduce again that any weighted stable set and any weighted minimizer is a horizontal half-space. 

Indeed, the only cases where $\text{int}(M)$ contains an entire hyperplane is when $M$ is a half-space or slab. In these situations the weighted stable sets are determined by half-spaces in $\rrn$ with boundary orthogonal or parallel to $\ptl M$. Next, we will show that the latter ones cannot be weighted isoperimetric regions. For this, we take a slab $E$ whose interior boundary $\Sg$ is a hyperplane parallel to $\ptl M$. We denote by $H_f$ the inner weighted mean curvature of $\Sg$. Consider the lower half-space $M_t:=M\times(-\infty,t)$ such that $V_f(M_t)=V_f(E)$. Suppose that $H_f\geq v'(t)$. By applying the Heintze-Karcher inequality with weights, see Morgan~\cite[Thm.~2]{morgandensity}, we obtain
\[
\frac{V_f(E)}{A_f(\Sg)}\leq\int_0^{r}e^{-H_f s-c s^2/2}\,ds<\int_0^{\infty}e^{-v'(t)\,s-c s^2/2}\,ds=\frac{V_f(M_t)}{A_f(\Om_t)},
\]
where $r:=\text{dist}(\Sg,\ptl M)$ and the last equality comes from a direct computation. This yields $A_f(\Om_t)<A_f(\Sg)$, so that $M_t$ is isoperimetrically better than $E$. In case of inequality $H_f\leq v'(t)$ we employ the same argument with $M-\overline{E}$ and $M-\overline{M}_t$ to obtain the comparison $A_f(\Om_t)\leq A_f(\Sg)$ with strict inequality unless $H_f=v'(t)$.
\end{example}

\begin{example}[Gaussian-type weights]
An interesting case of the previous example arises when $M:=\Om\times\rr$ is a convex cylinder of $\rrn$ endowed with a Gaussian weight $\gamma_c(p):=e^{-c|p|^2/2}$, $c>0$, or with a horizontal perturbation of the form $f(x,t):=\gamma_c(x,t)\,e^{\omega(x)}$, where $\omega$ is any smooth concave function on $\Om$. Our results extend previous ones of the author~\cite[Thms.~4.6 and~5.3]{rosales-gauss} for half-spaces and slabs with one-dimensional log-concave perturbations of $\gamma_c$. In particular, when $\omega(x_1,\ldots,x_n)=\sum_{i=1}^n\omega_i(x_i)$ for some $C^2$ concave functions $\omega_i$, the partitioning problem in $\Om\times\rr$ where $\Om:=\prod_{i=1}^n(a_i,b_i)$ and $-\infty\leq a_i<b_i\leq+\infty$ for any $i=1,\ldots,n$ was studied by Brock, Chiacchio and Mercaldo~\cite[Cor.~4.1]{bcm3}.
\end{example}

\subsection{The case $c=0$}
\label{subsec:zerocase}
\noindent

Here we consider a weighted cylinder $(M,g,f)$ with base space $\Om$, possibly with locally convex boundary, and weight $f(x,t):=e^{h(x)}\,e^{v(t)}$, where $h\in C^\infty(\Om)$ satisfies $\ric_h\geq 0$ and $v(t)$ is an affine function. This situation is very different from the one studied in Section~\ref{subsec:unique}. Note for instance that the horizontal half-spaces need not be weighted minimizers; indeed they can have infinite volume or perimeter. The existence of weighted isoperimetric regions is neither guaranteed in this case.

Let $E$ be a weighted stable set and denote by $\Sg$ the regular part of its interior boundary. If we suppose that $\Sg$ is disconnected or verifies the integral condition $\int_\Sg\escpr{\xi,N}\,da_f\neq 0$, then we deduce from Remarks~\ref{re:muchas} (i) that $\Sg$ is a totally geodesic hypersurface with $\ric_f(N,N)=0$ on $\Sg$ and $\text{II}(N,N)=0$ along $\ptl\Sg$. This is an extension of Theorem~\ref{th:stacyl} to the present setting.

As to the previous integral condition, if we proceed as in equation~\eqref{eq:divthweight} then we obtain
\[
\int_\Sg\escpr{\xi,N}\,da_f=-\int_E\escpr{\nabla\psi,\xi}\,dv_f=-\int_E(v'\circ\pi)\,dv_f,
\]
for any bounded set $E$ with smooth interior boundary $\Sg$. Hence, this integral vanishes if and only if $v'=0$ (as happens in the unweighted case). By this reason we will focus in the case $v'\neq 0$. 

In relation to our uniqueness results, if we further assume that $\ric_h>0$, or that $\ptl\Om$ is locally strictly convex and $\ptl\Sg\neq\emptyset$, then we can reason as in Corollary~\ref{cor:staunique} to conclude that any connected component of $\Sg$ is a horizontal slice. Since $v'\neq0$, the fact that $H_f$ is constant on $\Sg$ implies by the computations in Example~\ref{ex:horstable} that $\Sg$ is connected and $E$ is a horizontal half-space. As in Corollary~\ref{cor:isocyl} we get that any weighted minimizer is a horizontal half-space when $V_h(\Om)<\infty$. Since $V_f(E)<\infty$ the half-space will be a lower one or an upper one depending on the sign of $v'$. 

An interesting consequence is that bounded weighted isoperimetric regions with smooth interior boundary cannot exist in this setting. By Theorem~\ref{th:regularity} we infer that no bounded minimizers exist when $n\leq 6$, where $n$ is the dimension of $\Om$. In particular, this applies for the Euclidean weight $f(x,t):=e^t$, which is relevant since the associated $f$-minimal hypersurfaces (those with $H_f=0$) are the so-called \emph{translating solitons} of the mean curvature flow in $\rrn$.

\appendix
\section{Another proof of Corollary~\ref{cor:isocyl}}
\label{app:a}
\setcounter{theorem}{0}
\setcounter{equation}{0}  
\setcounter{subsection}{0}

Our rigidity properties for minimizers in weighted cylinders can be derived without using the stability condition further than the connectivity result in Corollary~\ref{cor:connected}. The argument is mainly based on the second variation formula and follows the proofs of \cite[Eq.~(3.41)]{bayle-thesis}, \cite[Thm.~3.2]{bayle-rosales} and \cite[Cor.~18.10]{gmt} with suitable modifications. This method also shows the isoperimetric property of horizontal half-spaces. We include here the details for the sake of completeness.

\begin{theorem}
\label{th:another}
Let $(M,g,f)$ be a weighted cylinder with base space $\Om$, possibly with locally convex boundary, and weight of the form
\[
f(x,t):=e^{h(x)}\,e^{v(t)},
\]
where $e^h$ is a smooth bounded function on $\Om$ with $\emph{Ric}_h\geq c>0$, and $e^v$ is an affine perturbation of the $c$-Gaussian soliton on $\rr$. Then, the horizontal half-spaces are weighted minimizers. Moreover, if $E$ is any weighted minimizer with regular and singular parts of its interior boundary denoted by $\Sg$ and $\Sg_0$ respectively, then $\Sg$ is a connected totally geodesic hypersurface, $\Sg_0=\emptyset$, $\emph{Ric}_f(N,N)=c$ on $\Sg$ and $\emph{II}(N,N)=0$ along $\ptl\Sg$.
\end{theorem}

\begin{proof}
From equation~\eqref{eq:riccidesc} we get $\ric_f\geq c$. Since $M$ is complete and $c>0$ we deduce that $V_f(M)<\infty$ by \cite[Cor.~4]{morgandensity}. This provides existence of weighted isoperimetric regions of any given volume, see \cite[Ch.~5]{gmt} and \cite[Sect.~A.2]{milman-invent}. Moreover, the weighted isoperimetric profile $I_{M,f}:(0,V_f(M))\to\rr^+_0$ defined by
\begin{equation}
\label{eq:isopro}
I_{M,f}(w):=\inf\{P_f(E)\,;\,E\sub M\text{ is a Borel set with }V_f(E)=w\}
\end{equation}
is a continuous function on $[0,V_f(M)]$ when extended to $0$ at the boundary points~\cite[Lem.~6.9]{milman-invent}.

Fix $w_0\in (0,V_f(M))$ and take a weighted minimizer $E$ with $V_f(E)=w_0$. By Theorem~\ref{th:regularity} the regular part $\Sg$ of $\overline{\ptl E\cap\text{int}(M)}$ is a smooth embedded hypersurface, possibly with boundary $\ptl\Sg=\Sg\cap\ptl M$. As $E$ is weighted stationary, the weighted mean curvature $H_f$ of $\Sg$ with respect to the unit normal $N$ pointing into $E$ is constant, and $\Sg$ meets $\ptl M$ orthogonally along $\ptl\Sg$, see Lemma~\ref{lem:varprop} (i). Moreover, we can apply Theorem~\ref{th:parabolicity} to infer that $\Sg$ is weighted parabolic. Thus, there is a sequence $\{\var_k\}_{k\in\nn}\sub C^\infty_0(\Sg)$ with $\var_k\neq 0$ and satisfying the properties in \eqref{eq:sequence}. From the stability of $E$ and Corollary~\ref{cor:connected} it follows that $\Sg$ is connected.

For any $k\in\nn$ we choose a smooth vector field $X_k$ with compact support on $M$, tangent along $\ptl M$, and such that $X_k=\varphi_k\,N$ on $\Sg$. The one-parameter flow $\{\phi_s\}_{s\in\rr}$ of $X_k$ produces a variation $E_s:=\phi_s(E)$ of $E$. Let $P_k(s):=P_f(E_s)$ and $V_k(s):=V_f(E_s)$ be the associated perimeter and volume functionals. Note that $P_f(s)=A_f(\Sg_s)$, where $\Sg_s=\phi_s(\Sg)$ for any $s\in\rr$. From the first variation formulas for the weighted area and volume in \cite[Lem.~3.2, Re.~3.1]{castro-rosales}, we have
\begin{equation}
\label{eq:michelle1}
P_k'(0)=-\int_\Sg H_f\,\varphi_k\,da_f, \quad 
V_k'(0)=-\int_\Sg\varphi_k\,da_f.
\end{equation}
As $V_k'(0)<0$ the function $V_k$ is a local diffeomorphism at $s=0$. In particular, there is an open interval $J_k$ containing $w_0$ where we can define $p_k:J_k\to\rr$ by $p_k(w):=P_k(V_k^{-1}(w))$. From equation \eqref{eq:isopro} and the fact that $E$ is a weighted isoperimetric region, it is clear that
\[
I_{M,f}(w)\leq p_k(w),\qquad I_{M,f}(w_0)=p_k(w_0)=P_f(E).
\]
Hence, if we consider the weak second derivative of $I_{M,f}$ at $w_0$ given by
\[
I''_{M,f}(w_0):=\limsup_{h\to 0^+}\,
\frac{I_{M,f}(w_0+h)+I_{M,f}(w_0-h)-2I_{M,f}(w_0)}{h^2},
\]
then
\begin{equation}
\label{eq:michelle2}
I''_{M,f}(w_0)\leq p_k''(w_0), \quad\text{for any } k\in\nn,
\end{equation}
where $p''_k(w_0)$ is the usual second derivative of $p_k$ at $w_0$. Let us compute this derivative. On the one hand, by using \eqref{eq:michelle1} we get
\[
p_k'(w_0)=\frac{P_k'(0)}{V_k'(0)}=H_f, \quad\text{for any } k\in\nn.
\]
On the other hand, from the second variation formula in \cite[Prop.~3.5]{castro-rosales} we obtain
\begin{align*}
p_k''(w_0)=\frac{(P_k-H_f\,V_k)''(0)}{V'_k(0)^2}=\frac{\indo_f(\var_k,\var_k)}{(\int_\Sg\var_k\,da_f)^2}, \quad\text{for any } k\in\nn,
\end{align*}
where $\indo_f$ is the weighted index form in \eqref{eq:index1}. So, by taking $\limsup$ in \eqref{eq:michelle2} when $k\to\infty$, having in mind \eqref{eq:sequence} and applying Fatou's lemma and the dominated convergence theorem, we deduce
\[
I_{M,f}''(w_0)\leq-\frac{\int_\Sg\big(\text{Ric}_f(N,N)+|\sg|^2\big)\,da_f
+\int_{\ptl\Sg}\text{II}(N,N)\,dl_f}{P_f(E)^2}\leq\frac{-c}{I_{M,f}(w_0)},
\]
where in the last inequality we have employed $\text{Ric}_f(N,N)\geq c$, $|\sg|^2\geq 0$ and $\text{II}(N,N)\geq 0$. 

At summarizing, we have shown that $I_{M,f}$ satisfies the differential inequality
\begin{equation}
\label{eq:diffineq1}
I''_{M,f}\,(w)\leq\frac{-c}{I_{M,f}(w)}, \quad\text{for any } w\in (0,V_f(M)).
\end{equation}
Moreover, if equality holds for some $w_0\in(0,V_f(M))$ and $E$ is any weighted minimizer with $V_f(E)=w_0$, then $\Sg$ is totally geodesic, $\ric_f(N,N)=c$ on $\Sg$ and $\text{II}(N,N)=0$ along $\ptl\Sg$. The fact that $|\sg|^2=0$ on $\Sg$ entails that $\Sg_0=\emptyset$ by Theorem~\ref{th:regularity}.

Now consider the lower horizontal half-space $M_t:=\Om\times(-\infty,t)$, whose interior boundary is $\Om_t:=\Om\times\{t\}$. Since $\ric_h\geq c>0$ we get $V_h(\Om)<\infty$, so that $\Om_t$ has finite weighted area
\[
A_f(t):=A_f(\Om_t)=V_h(\Om)\,e^{v(t)}.
\]
If we denote $V_f(t):=V_f(M_t)$, then the coarea formula applied with $\pi(x,t):=t$ yields
\[
V_f(t)=\int_{-\infty}^t A_f(s)\,ds.
\]
From here it is straightforward to check that the associated profile $G_f:[0,V_f(M)]\to\rr$ defined by $G_f(w):=A_f(V_f^{-1}(w))$ is a smooth function vanishing at the boundary, and satisfying
\begin{equation}
\label{eq:diffineq2}
G''_{f}\,(w)=\frac{-c}{G_f(w)}, \quad\text{for any } w\in (0,V_f(M)).
\end{equation}

It is clear that $I_{M,f}\leq G_f$ on $[0,V_f(M)]$ with equality at the boundary points. By continuity the function $I_{M,f}-G_f:[0,V_f(M)]\to\rr$ achieves its minimum at $w_0\in [0,V_f(M)]$. Suppose that there was $w\in(0,V_f(M))$ with $I_{M,f}(w)<G_f(w)$. In such a case $I_{M,f}(w_0)<G_f(w_0)$, so that $w_0\in(0,V_f(M))$ and $(I_{M,f}-G_f)''(w_0)\geq 0$. On the other hand, by taking into account \eqref{eq:diffineq1} and \eqref{eq:diffineq2}, we would obtain
\[
(I_{M,f}-G_f)''(w_0)\leq\frac{c}{G_f(w_0)}-\frac{c}{I_{M,f}(w_0)}<0,
\]
a contradiction. This proves that $I_{M,f}=G_f$ and so, any lower horizontal half-space is a weighted isoperimetric region. From \eqref{eq:diffineq2} we infer that equality holds in \eqref{eq:diffineq1} on $(0,V_f(M))$. Thus, we conclude the announced rigidity properties for any weighted minimizer.
\end{proof}

\providecommand{\bysame}{\leavevmode\hbox to3em{\hrulefill}\thinspace}
\providecommand{\MR}{\relax\ifhmode\unskip\space\fi MR }
\providecommand{\MRhref}[2]{%
  \href{http://www.ams.org/mathscinet-getitem?mr=#1}{#2}
}
\providecommand{\href}[2]{#2}


\begin{thebibliography}{10}

\bibitem{alencar-rocha}
H.~Alencar and A.~Rocha, \emph{Stability and geometric properties of constant
  weighted mean curvature hypersurfaces in gradient {R}icci solitons}, Ann.
  Global Anal. Geom. \textbf{53} (2018), no.~4, 561--581. \MR{3803340}

\bibitem{agm}
S.~Artstein-Avidan, A.~Giannopoulos, and V.~D. Milman, \emph{Asymptotic
  geometric analysis. {P}art {I}}, Mathematical Surveys and Monographs, vol.
  202, American Mathematical Society, Providence, RI, 2015. \MR{3331351}

\bibitem{be}
D.~Bakry and M.~{\'E}mery, \emph{Diffusions hypercontractives}, S\'eminaire de
  probabilit\'es, {XIX}, 1983/84, Lecture Notes in Math., vol. 1123, Springer,
  Berlin, 1985, pp.~177--206. \MR{MR889476 (88j:60131)}

\bibitem{bl}
D.~Bakry and M.~Ledoux, \emph{L\'evy-{G}romov's isoperimetric inequality for an
  infinite-dimensional diffusion generator}, Invent. Math. \textbf{123} (1996),
  no.~2, 259--281. \MR{MR1374200 (97c:58162)}

\bibitem{chiara}
F.~Barthe, C.~Bianchini, and A.~Colesanti, \emph{Isoperimetry and stability of
  hyperplanes for product probability measures}, Ann. Mat. Pura Appl. (4)
  \textbf{192} (2013), no.~2, 165--190. \MR{3035134}

\bibitem{bayle-thesis}
V.~Bayle, \emph{Propri\'et\'es de concavit\'e du profil isop\'erim\'etrique et
  applications}, Ph.D. thesis, Institut Fourier (Grenoble), 2003.

\bibitem{bayle-paper}
\bysame, \emph{A differential inequality for the isoperimetric profile}, Int.
  Math. Res. Not. (2004), no.~7, 311--342. \MR{2041647 (2005a:53050)}

\bibitem{bayle-rosales}
V.~Bayle and C.~Rosales, \emph{Some isoperimetric comparison theorems for
  convex bodies in {R}iemannian manifolds}, Indiana Univ. Math. J. \textbf{54}
  (2005), no.~5, 1371--1394. \MR{2177105 (2006f:53040)}

\bibitem{bcm}
F.~Brock, F.~Chiacchio, and A.~Mercaldo, \emph{A class of degenerate elliptic
  equations and a {D}ido's problem with respect to a measure}, J. Math. Anal.
  Appl. \textbf{348} (2008), no.~1, 356--365. \MR{2449353 (2010h:35146)}

\bibitem{bcm3}
\bysame, \emph{An isoperimetric inequality for {G}auss-like product measures},
  J. Math. Pures Appl. (9) \textbf{106} (2016), no.~2, 375--391. \MR{3515307}

\bibitem{caffarelli}
L.~A. Caffarelli, \emph{Monotonicity properties of optimal transportation and
  the {FKG} and related inequalities}, Comm. Math. Phys. \textbf{214} (2000),
  no.~3, 547--563. \MR{1800860 (2002c:60029)}

\bibitem{cao-zhou}
H.-D. Cao and D.~Zhou, \emph{On complete gradient shrinking {R}icci solitons},
  J. Differential Geom. \textbf{85} (2010), no.~2, 175--185. \MR{2732975}

\bibitem{castro-rosales}
K.~Castro and C.~Rosales, \emph{Free boundary stable hypersurfaces in manifolds
  with density and rigidity results}, J.~Geom.~Phys. \textbf{79} (2014),
  14--28.

\bibitem{cmz2}
X~Cheng, T.~Mejia, and D.~Zhou, \emph{Simons-type equation for {$f$}-minimal
  hypersurfaces and applications}, J. Geom. Anal. \textbf{25} (2015), no.~4,
  2667--2686. \MR{3427142}

\bibitem{colding-minicozzi}
T.~H. Colding and W.~P. Minicozzi~II, \emph{Generic mean curvature flow {I}:
  generic singularities}, Ann. of Math. (2) \textbf{175} (2012), no.~2,
  755--833. \MR{2993752}

\bibitem{dcriem}
M.~P. do~Carmo, \emph{Riemannian geometry}, Mathematics: Theory \&
  Applications, Birkh\"auser Boston Inc., Boston, MA, 1992, Translated from the
  second Portuguese edition by Francis Flaherty. \MR{MR1138207 (92i:53001)}

\bibitem{calibrations}
T.~H. Doan, \emph{Some calibrated surfaces in manifolds with density}, J. Geom.
  Phys. \textbf{61} (2011), no.~8, 1625--1629. \MR{2802497 (2012e:53091)}

\bibitem{giusti}
E.~Giusti, \emph{Minimal surfaces and functions of bounded variation},
  Monographs in Mathematics, vol.~80, Birkh{\"a}user Verlag, Basel, 1984.
  \MR{775682 (87a:58041)}

\bibitem{go-ma-ta}
E.~Gonzalez, U.~Massari, and I.~Tamanini, \emph{On the regularity of boundaries
  of sets minimizing perimeter with a volume constraint}, Indiana Univ. Math.
  J. \textbf{32} (1983), no.~1, 25--37. \MR{684753 (84d:49043)}

\bibitem{grigoryan}
A.~Grigor'yan, \emph{Analytic and geometric background of recurrence and
  non-explosion of the {B}rownian motion on {R}iemannian manifolds}, Bull.
  Amer. Math. Soc. (N.S.) \textbf{36} (1999), no.~2, 135--249. \MR{1659871
  (99k:58195)}

\bibitem{grigoryan2}
\bysame, \emph{Heat kernels on weighted manifolds and applications}, The
  ubiquitous heat kernel, Contemp. Math., vol. 398, Amer. Math. Soc.,
  Providence, RI, 2006, pp.~93--191. \MR{2218016 (2007a:58028)}

\bibitem{gri-masa}
A.~Grigor'yan and J.~Masamune, \emph{Parabolicity and stochastic completeness
  of manifolds in terms of the {G}reen formula}, J. Math. Pures Appl. (9)
  \textbf{100} (2013), no.~5, 607--632. \MR{3115827}

\bibitem{gromov-GAFA}
M.~Gromov, \emph{Isoperimetry of waists and concentration of maps}, Geom.
  Funct. Anal. \textbf{13} (2003), no.~1, 178--215. \MR{MR1978494
  (2004m:53073)}

\bibitem{gruter}
M.~Gr{{\"u}}ter, \emph{Boundary regularity for solutions of a partitioning
  problem}, Arch. Rational Mech. Anal. \textbf{97} (1987), no.~3, 261--270.
  \MR{862549 (87k:49050)}

\bibitem{hamilton}
R.~S. Hamilton, \emph{Three-manifolds with positive {R}icci curvature}, J.
  Differential Geometry \textbf{17} (1982), no.~2, 255--306. \MR{664497}

\bibitem{contraction}
Y.-H. Kim and E.~Milman, \emph{A generalization of {C}affarelli's contraction
  theorem via (reverse) heat flow}, Math. Ann. \textbf{354} (2012), no.~3,
  827--862. \MR{2983070}

\bibitem{kolesnikov-milman2}
A.~V. Kolesnikov and E.~Milman, \emph{Riemannian metrics on convex sets with
  applications to {P}oincar\'{e} and log-{S}obolev inequalities}, Calc. Var.
  Partial Differential Equations \textbf{55} (2016), no.~4, Art. 77, 36.
  \MR{3514409}

\bibitem{markov}
M.~Ledoux, \emph{The geometry of {M}arkov diffusion generators}, Ann. Fac. Sci.
  Toulouse Math. (6) \textbf{9} (2000), no.~2, 305--366, Probability theory.
  \MR{1813804 (2002a:58045)}

\bibitem{lich1}
A.~Lichnerowicz, \emph{Vari{\'e}t{\'e}s riemanniennes {\`a} tenseur {C} non
  n{\'e}gatif}, C. R. Acad. Sci. Paris S{\'e}r. A-B \textbf{271} (1970),
  A650--A653. \MR{0268812 (42 \#3709)}

\bibitem{lich2}
\bysame, \emph{Vari{\'e}t{\'e}s k{\"a}hl{\'e}riennes {\`a} premi{\`e}re classe
  de {C}hern non negative et vari{\'e}t{\'e}s riemanniennes {\`a} courbure de
  {R}icci g{\'e}n{\'e}ralis{\'e}e non negative}, J. Differential Geom.
  \textbf{6} (1971/72), 47--94. \MR{0300228 (45 \#9274)}

\bibitem{maeda}
M.~Maeda, \emph{Volume estimate of submanifolds in compact {R}iemannian
  manifolds}, J. Math. Soc. Japan \textbf{30} (1978), no.~3, 533--551.
  \MR{500722}

\bibitem{maggi}
F.~Maggi, \emph{Sets of finite perimeter and geometric variational problems},
  Cambridge Studies in Advanced Mathematics, vol. 135, Cambridge University
  Press, Cambridge, 2012, An introduction to geometric measure theory.
  \MR{2976521}

\bibitem{mcgonagle-ross}
M.~McGonagle and J.~Ross, \emph{{The hyperplane is the only stable, smooth
  solution to the isoperimetric problem in Gaussian space}}, {Geom. Dedicata}
  \textbf{178} (2015), 277--296 (English).

\bibitem{milman-invent}
E.~Milman, \emph{On the role of convexity in isoperimetry, spectral gap and
  concentration}, Invent. Math. \textbf{177} (2009), no.~1, 1--43. \MR{2507637
  (2010j:28004)}

\bibitem{milman}
\bysame, \emph{Sharp isoperimetric inequalities and model spaces for the
  curvature-dimension-diameter condition}, J. Eur. Math. Soc. (JEMS)
  \textbf{17} (2015), no.~5, 1041--1078. \MR{3346688}

\bibitem{morgan-reg}
F.~Morgan, \emph{Regularity of isoperimetric hypersurfaces in {R}iemannian
  manifolds}, Trans. Amer. Math. Soc. \textbf{355} (2003), no.~12, 5041--5052.
  \MR{1997594 (2004j:49066)}

\bibitem{morgandensity}
\bysame, \emph{Manifolds with density}, Notices Amer. Math. Soc. \textbf{52}
  (2005), no.~8, 853--858. \MR{MR2161354 (2006g:53044)}

\bibitem{gmt}
\bysame, \emph{Geometric measure theory. {A} beginner's guide}, fourth ed.,
  Elsevier/Academic Press, Amsterdam, 2009. \MR{2455580 (2009i:49001)}

\bibitem{oneill}
B.~O'Neill, \emph{Semi-{R}iemannian geometry}, Pure and Applied Mathematics,
  vol. 103, Academic Press, Inc. [Harcourt Brace Jovanovich, Publishers], New
  York, 1983, With applications to relativity. \MR{719023 (85f:53002)}

\bibitem{petersen-wylie}
P.~Petersen and W.~Wylie, \emph{On gradient {R}icci solitons with symmetry},
  Proc. Amer. Math. Soc. \textbf{137} (2009), no.~6, 2085--2092. \MR{2480290}

\bibitem{rosales-gauss}
C.~Rosales, \emph{Isoperimetric and stable sets for log-concave perturbations
  of {G}aussian measures}, Anal. Geom. Metr. Spaces \textbf{2} (2014),
  328--358. \MR{3290382}

\bibitem{rcbm}
C.~Rosales, A.~Ca{\~n}ete, V.~Bayle, and F.~Morgan, \emph{On the isoperimetric
  problem in {E}uclidean space with density}, Calc. Var. Partial Differential
  Equations \textbf{31} (2008), no.~1, 27--46. \MR{2342613 (2008m:49212)}

\bibitem{sz}
P.~Sternberg and K.~Zumbrun, \emph{On the connectivity of boundaries of sets
  minimizing perimeter subject to a volume constraint}, Comm. Anal. Geom.
  \textbf{7} (1999), no.~1, 199--220. \MR{1674097 (2000d:49062)}

\bibitem{welsh1}
D.~J. Welsh, \emph{Manifolds that admit parallel vector fields}, Illinois J.
  Math. \textbf{30} (1986), no.~1, 9--18. \MR{822382}

\bibitem{welsh2}
\bysame, \emph{On the existence of complete parallel vector fields}, Proc.
  Amer. Math. Soc. \textbf{97} (1986), no.~2, 311--314. \MR{835888}

\end{thebibliography}
\end{document}